\documentclass{amsart}
\usepackage{latexsym}
\usepackage{amsmath}
\usepackage{amssymb}
\usepackage{amsthm}
\usepackage{amscd}
\usepackage{enumerate} 
\usepackage{amssymb} 
\usepackage{mathrsfs}
\usepackage[all]{xy}
\usepackage{color}
\usepackage{graphicx}
\usepackage{mathtools}
\usepackage{comment}
\usepackage{multirow}
\usepackage
{hyperref}
\usepackage[foot]{amsaddr}
\usepackage{bm}
\hypersetup{colorlinks=true}
\usepackage{tikz}
\usepackage{diagbox}
\usetikzlibrary{positioning}
\numberwithin{equation}{section}

\title{Non-log liftable log del Pezzo surfaces of rank one in characteristic five}

\author{Masaru Nagaoka}
\email{masaru.nagaoka@gakushuin.ac.jp}
\address{Gakushuin University, 1-5-1 Mejiro, Toshima-ku, Tokyo 171-8588, Japan}

\def\phi{\varphi}
\def\epsilon{\varepsilon}

\def\mapsto{\longmapsto}

\def\log{\operatorname{log}}

\def\Spec{\operatorname{Spec}}

\def\Dyn{\operatorname{Dyn}}
\def\Sing{\operatorname{Sing}}

\newcommand{\Q}{\mathbb{Q}} 
\newcommand{\C}{\mathbb{C}} 
 
\newcommand{\Z}{\mathbb{Z}}
\renewcommand{\AA}{\mathbb{A}}
\newcommand{\PP}{\mathbb{P}}

\newcommand{\sO}{\mathcal{O}}

\newcommand{\mc}{\mathcal}
\newcommand{\wt}{\widetilde}

\theoremstyle{plain}
\newtheorem{thm}{Theorem}[section] 
\newtheorem{cor}[thm]{Corollary}
\newtheorem{prop}[thm]{Proposition}

\newtheorem{lem}[thm]{Lemma}
\theoremstyle{definition} 
\newtheorem{defn}[thm]{Definition}

\newtheorem{eg}[thm]{Example} 

\theoremstyle{remark}
\newtheorem{rem}[thm]{Remark}

\newtheorem{defn and notation}[thm]{Definition and Notation}
\newtheorem*{notation}{Notation}

\keywords{Log del Pezzo surfaces; Liftability to the ring of Witt vectors; Positive characteristic.}
\subjclass[2020]{Primary 14J26, 14D15; Secondary 14G17, 14J45}

\begin{document}
\tolerance = 9999

\begin{abstract}
    Building upon the classification by Lacini [arXiv:2005.14544], we determine the isomorphism classes of log del Pezzo surfaces of rank one over an algebraically closed field of characteristic five either which are not log liftable over the ring of Witt vectors or whose singularities are not feasible in characteristic zero.
    We also show that the Kawamata-Viehweg vanishing theorem for ample $\Z$-Weil divisors holds for log del Pezzo surfaces of rank one in characteristic five if those singularities  are feasible in characteristic zero.
\end{abstract}

\maketitle
\markboth{Masaru Nagaoka}{Non-log liftable log del Pezzo surfaces in characteristic five}

\tableofcontents

\section{Introduction}

In \cite{KN1}, Kawakami and the author investigated Du Val del Pezzo surfaces over an algebraically closed field $k$ of characteristic $p>0$.
As a result, we revealed when such a surface has several pathological properties such as non-existence of Du Val del Pezzo surfaces over the field of complex numbers $\C$ with the same Dynkin type and the same Picard rank, non-existence of smooth anti-canonical members, the failure of the Kawamata-Viehweg vanishing theorem for ample $\Z$-Weil divisors, and the non-log liftability over the ring of Witt vectors $W(k)$ (see Definition \ref{d-liftable}).

Around the same time, Lacini \cite{Lac} gave the classification of log del Pezzo surfaces of rank one over $k$ of characteristic $p >3$ following Keel and M\textsuperscript{c}Kernan's idea \cite{KM}.
As a result, he showed that every log del Pezzo surfaces of rank one over $k$ is log liftable over $W(k)$ when we add the assumption $p \neq 5$.

The aim of this paper is to investigate pathological phenomena as above on log del Pezzo surfaces of rank one over $k$ of characteristic $p=5$.
To this end, we define the following notation.
For the definition and terminology of Dynkin types, we refer to $\S \ref{subsec:Dynkin}$.

\begin{defn}\label{notation}
For a log del Pezzo surface $S$ of index $r$ over an algebraically closed field $k$ of characteristic $p>0$, we say that $S$ satisfies:
\begin{itemize}
    \item (ND) if there exists no log del Pezzo surface $S_{\C}$ over the field of complex numbers $\C$ with the same Dynkin type and the same Picard rank as $S$.
    \item (NB) if the members of $|-rK_S|$ are all singular. 
    \item (NK) if $H^1(S, \sO_S(-A)) \neq 0$ for some ample $\Z$-Weil divisor $A$ on $S$.
    \item (NL) if $S$ is not log liftable over $W(k)$.
\end{itemize}
\end{defn}

Note that $S_\C$ as in the condition (ND) must have the same anti-canonical volume as $S$ since log del Pezzo surfaces are rational.
Furthermore, compared to the case of Du Val del Pezzo surfaces \cite[Definition 1.2]{KN1}, the property (NB) does not express pathological features very often since, to the best of the author's knowledge, it is unknown whether $|-rK_S|$ has a smooth member even in characteristic zero.

The main results of this paper consist of two theorems.
One is Theorem \ref{thm:main}, which determines log del Pezzo surfaces of rank one satisfying (NL) or (ND) when $p=5$.

\begin{thm}\label{thm:main}
Let $S$ be a log del Pezzo surface of rank one over an algebraically closed field $k$ of characteristic five.
Then the following hold.
\begin{enumerate}
    \item[\textup{(1)}] $S$ is not log liftable over the ring of Witt vectors $W(k)$ if and only if $\Dyn (S) =2[3,2]+[3]+[2]+[2^4]$, $[4,2]+[3,2^5]+[3,2]+[2]$, $[2,4]+[2,3,2^2]+[3]+[2^4]$, or $2[2^4]+(\dagger)$, where $(\dagger)$ is listed in Table \ref{tab:dagger}.
    \item[\textup{(2)}] The isomorphism class of $S$ is uniquely determined by its Dynkin type except when $\Dyn (S)=2[2^4]+ [2^n]+[2+n; [2], [3], [5]]$ for some $n \geq 0$.
    In these cases, the isomorphism classes of del Pezzo surfaces correspond to the closed points of $\PP^1_k \setminus \{0, 1, \infty\}$ for each $n \geq 0$.
    \item[\textup{(3)}] $S$ is not log liftable over $W(k)$ but there exists a log del Pezzo surface of rank one over $\C$ with the same Dynkin type as $S$ if and only if $\Dyn (S) = 2[2^4]+[3]$ or $2[2^4]+[2, 4]$.
\end{enumerate}
\end{thm}

\begin{table}[ht]
    \centering
\begin{align*}
    &[3], [2, 4], [2]+[3]+[5], [2^n]+[2+n; [2], [3], [5]], \\
    &[2^n, 3]+[3, 2+n, 5], [2^n, 4]+[2, 2+n, 5], [2^n, 6]+[2, 2+n, 3], \\
    &[4]+[2; [2], [3], [5]], [3, 2^{m-1}, 3]+[2+m; [2], [3], [5]], \\
    &[4+l]+[2; [2], [2^l], [5]], [2+l, 2^{m-1}, 4]+[2+m; [2], [2^l], [5]]\text{ with } 1 \leq l \leq 2, \\
    &[2, 5]+[2;[2], [3], [5]], [2, 3, 2^{m-1}, 4]+[2+m; [2], [3], [5]],\\
    &[6+l]+[2; [2], [3], [2^l]], [2+l, 2^{m-1}, 6]+[2+m;[2], [3], [2^l]] \text{ with } 1 \leq l \leq 4, \\ 
    &[2^l, 7]+[2; [2], [3], [2+l]], [2^l, 3, 2^{m-1}, 6]+[2+m; [2], [3], [2+l]] \text{ with } 1 \leq l \leq 3, \\
    &[3, 7]+[2; [2], [3], [3, 2]], [3, 3, 2^{m-1}, 6]+[2+m; [2], [3], [3, 2]], \\
    &[2, 8]+[2; [2], [3], [2, 3]], [2, 4, 2^{m-1}, 6]+[2+m; [2], [3], [2, 3]].
    \end{align*}
    \caption{The list of $(\dagger)$ where $n \geq 0$ and $m \geq 1$
    }
    \label{tab:dagger}
\end{table}

The other is Theorem \ref{thm:main2}, which shows that (NK) $\Rightarrow$ (ND) $\Rightarrow$ (NL) when the rank is one and $p=5$. 
The latter implication also holds in $p > 0$.

\begin{thm}\label{thm:main2}
Let $S$ be a log del Pezzo surface of rank one over an algebraically closed field $k$ of characteristic $p>0$.
Then the following hold.
\begin{enumerate}
\item[\textup{(1)}] If $S$ is log liftable over $W(k)$, then there exists a log del Pezzo surface of rank one over $\C$ with the same Dynkin type as $S$.
\item[\textup{(2)}] If $p=5$ and there exists a log del Pezzo surface of rank one over $\C$ with the same Dynkin type as $S$, then $H^1(S, \sO_S(-A))=0$ for every ample $\Z$-Weil divisor $A$.
\end{enumerate}
\end{thm}

We can prove the assertion (1) without Theorem \ref{thm:main}.
However, the assertion (2) follows from Theorem \ref{thm:main} and detailed observations of log del Pezzo surfaces with Dynkin type $2[2^4]+[3]$ or $2[2^4]+[2,4]$ so far.
Note that (NK) $\Rightarrow$ (ND) also holds for Du Val del Pezzo surfaces in any characteristic as a consequence of classifications \cite[Theorem 1.7]{KN1}.
What is still lacking is a proof of (NK) $\Rightarrow$ (ND) without using the classification of del Pezzo surfaces satisfying (NL).

This paper is structured as follows.
In Section \ref{sec:pre}, we set up terminology of klt singularities, recall hunt steps, and compile facts on Du Val del Pezzo surfaces with Dynkin type $2[2^4]$.
In Section \ref{sec:implication},  we give implications between pathological phenomena on log del Pezzo surfaces of rank one.
In particular, we prove Theorem \ref{thm:main2} (1).
In Section \ref{sec:notiger}, we construct non-log liftable log del Pezzo surfaces in Example \ref{eg:Lac4.8cex-1}
, which were erroneously omitted in \cite[\S 6.1]{Lac}
when we submitted this paper.
In Section \ref{sec:Lac}, we prove a refinement of \cite[Theorem 6.2]{Lac} (= Theorem \ref{Lac6.2}) and an alternative of \cite[Theorem 6.25]{Lac} (= Theorem \ref{Lac6.22}), which make it easy to seek a necessary condition for log del Pezzo surfaces of rank one to satisfy (NL).
In Section \ref{sec:main}, we prove Theorem \ref{thm:main} and Theorem \ref{thm:main2} (2) as follows.
We can calculate easily the list of Dynkin types of non-log liftable log del Pezzo surfaces of rank one by Example \ref{eg:Lac4.8cex-1} and Theorem \ref{Lac6.22}.
On the other hand, from the Bogomolov bound, we see that any Dynkin types in the list except $2[2^4]+[3], 2[2^4]+[2,4]$ and $2[2^4]+[2;[2], [3], [5]]$ are not feasible over $\C$.
Running hunt steps, we show that the last case is also not feasible over $\C$.
For the rest two cases, which are feasible over $\C$, we prove the non-log liftability by assuming the opposite.
For a log lift of such a log del Pezzo surface, we contract a certain extremal ray to get a log lift of the Du Val del Pezzo surface with Dynkin type $2[2^4]$.
Since all the singular anti-canonical member of this Du Val del Pezzo surface are defined over the fractional field of $W(k)$, we deduce that $\sqrt{5}$ must be contained in $W(k)$ to get a contradiction.
Thus we prove Theorem \ref{thm:main}.
To show Theorem \ref{thm:main2} (2), we calculate the first cohomology of anti-ample $\Z$-Weil divisors of small canonical degree explicitly.

\begin{notation}
We work over an algebraically closed field $k$ of characteristic $p > 0$ unless otherwise stated.
A \textit{variety} means an integral separated scheme of finite type over $k$.
A \textit{curve} (resp.\ a \text{surface}) means a variety of dimension one (resp.\ two).
A \textit{log pair} $(S, \Delta)$ is the pair of a normal projective variety $S$ and an effective $\Q$-Weil divisor $\Delta$ such that $K_S+\Delta$ is $\Q$-Cartier.
We say that a $\Q$-Weil divisor $\Delta$ is a \textit{boundary} if the irreducible decomposition $\Delta = \sum_i c_i D_i$ satisfies $0 \leq c_i \leq 1$.
We say that a normal surface is \textit{Du Val} if it has only canonical singularities. 
A \textit{log del Pezzo surface} is a projective surface with only klt singularities whose anti-canonical divisor is ample.
We often call the Picard rank of a log del Pezzo surface simply as the \textit{rank}.
Throughout this paper, we also use the following notation:
\begin{itemize}
\item $W(k)$ : the ring of Witt vectors of $k$.
\item $E_f$: the reduced exceptional divisor of a birational morphism $f$.
\item $S^\circ$: the smooth locus of a surface $S$.
\item $\Delta(\Gamma)$: the absolute value of the determinant of the intersection matrix of the weighted graph $\Gamma$.
\item $p_a(C)$: the arithmetic genus of a projective curve $C$.
\end{itemize}
\end{notation}

\section{Preliminaries}\label{sec:pre}

\subsection{Dynkin types}\label{subsec:Dynkin}

In this subsection, we set up notation and terminology of klt singularities of dimension two.

\begin{defn}
For integers $c, l_1, \ldots, l_i, m_1, \ldots, m_j, n_1, \ldots n_k \in \Z_{\geq 2}$, the symbol $[n_1, \ldots, n_k]$ (resp.\ $[c; [l_1, \ldots, l_i], [m_1, \ldots, m_j], [n_1, \ldots, n_k]]$) stands for the weighted dual graph as in Figure \ref{fig:chain} (resp.\ Figure \ref{fig:star}).
\end{defn}

\begin{figure}[ht] 
		\centering
		\begin{tikzpicture}
		
		\draw[] (0,0)circle(2pt);
		\draw[] (1,0)circle(2pt);
		\draw[] (2,0)circle(2pt);
		\node(a)at(3,0){$\cdots$};
		\draw[] (4,0)circle(2pt);
		
		\draw[thin ](0.2,0)--(0.8,0);
		\draw[thin ](1.2,0)--(1.8,0);
		\draw[thin ](2.2,0)--(2.6,0);
		\draw[thin ](3.4,0)--(3.8,0);

		\node(a)at(0,0.45){$-n_1$};
		\node(a)at(1,0.45){$-n_2$};
		\node(a)at(2,0.45){$-n_3$};
		\node(a)at(4,0.45){$-n_k$};

		\end{tikzpicture}
		\caption{The weighted graph $[n_1, n_2, \ldots, n_k]$.
		}\label{fig:chain}
		\centering
		\begin{tikzpicture}
		
		\draw[ ] (-3,0)circle(2pt);
		\node(a)at(-2,0){$\cdots$};
		\draw[] (-1,0)circle(2pt);
		\draw[] (0,0)circle(2pt);
		\draw[] (1,0)circle(2pt);
		\node(a)at(2,0){$\cdots$};
		\draw[] (3,0)circle(2pt);
		\draw[] (0,-1)circle(2pt);
		\node(a)at(0,-2){$\vdots$};
		\draw[] (0,-3)circle(2pt);
		
		\draw[thin ](-2.8,0)--(-2.4,0);
		\draw[thin ](-1.6,0)--(-1.2,0);
		\draw[thin ](-0.8,0)--(-0.2,0);
		\draw[thin ](0.2,0)--(0.8,0);
		\draw[thin ](1.2,0)--(1.6,0);
		\draw[thin ](2.4,0)--(2.8,0);
		\draw[thin ](0,-0.2)--(0,-0.8);
		\draw[thin ](0,-1.2)--(0,-1.8);
		\draw[thin ](0,-2.4)--(0,-2.8);

		\node(a)at(-3,0.45){$-n_k$};
		\node(a)at(-1,0.45){$-n_1$};
		\node(a)at(0,0.45){$-c$};
		\node(a)at(1,0.45){$-m_1$};
		\node(a)at(3,0.45){$-m_j$};
		\node(a)at(-0.45,-1){$-l_1$};
		\node(a)at(-0.45,-3){$-l_i$};

		\end{tikzpicture}
		\caption{The weighted graph $[c; [l_1, \ldots, l_i], [m_1, \ldots, m_j],$ $[n_1, \ldots, n_k]]$.
		}\label{fig:star}
	\end{figure}

\begin{defn}\label{Dynkin}
Let $S$ be a surface with klt singularities and $\sigma \colon \wt S \to S$ the minimal resolution.
The \textit{Dynkin type} of $S$, denoted by $\Dyn(S)$, is defined to be the weighted dual graph of the reduced $\sigma$-exceptional divisor.
\end{defn}

For example, we write $\Dyn(S)=2[2^4]+[2;[2],[3],[5]]$ if the weighted dual graph of the reduced $\sigma$-exceptional divisor consists of two connected components isomorphic to $[2^4]$ and one connected component isomorphic to $[2;[2], [3], [5]]$.
Recall that we can express the Dynkin type of each klt singularity by $[n_1, \ldots, n_k]$ or $[c; [l_1, \ldots, l_i], [m_1, \ldots, m_j], [n_1, \ldots, n_k]]$ (see \cite[\S 3]{Kol92} for more details).

\subsection{Hunt steps}

In this subsection, we recall certain Sarkisov links called hunt steps, which are used in \cite{KM} and \cite{Lac}.

\begin{defn}
Let $(X, \Delta)$ be a log pair. 
Let $f \colon Y \to X$ be a birational morphism with $Y$ normal and $\Gamma$ the log pullback of $\Delta$ via $f$.
For a prime divisor $E$ on $Y$, we define the \textit{coefficient} $e(E; X, \Delta)$ of $E$ with respect to the pair $(X, \Delta)$ to be the multiplicity of $\Gamma$ along $E$.
\end{defn}

\begin{defn}[{\cite[8.0.2 Definition]{KM}}]
Let $(X, \Delta)$ be a log pair with $\Delta$ a boundary and $\Delta = \sum_{i=1}^n a_i D_i$ the irreducible decomposition.
We say that $(X, \Delta)$ is \textit{flush} if $e(E; X, \Delta) < \min\{a_1, \ldots, a_n, 1\}$ for all exceptional divisors $E$ over $X$.
\end{defn}

\begin{defn}[{\cite[1.13 Definition]{KM} and \cite[Definitions 3.10 and 6.1]{Lac}}]
Let $(X, \Delta)$ be a log pair and $f \colon Y \to X$ be a birational morphism with $Y$ normal.
We say that $(X, \Delta)$ \textit{has a tiger} in $Y$ if there exists an effective $\Q$-Cartier divisor $\alpha$ on $X$ such that
\begin{enumerate}
    \item[\textup{(1)}] $K_X+\Delta+\alpha$ is numerically trivial and
    \item[\textup{(2)}] there exists a prime divisor $E$ on $Y$ such that $e(E; X, \Delta+\alpha) \geq 1$.
\end{enumerate}
Any such divisor $E$ is called a \textit{tiger} for $(X, \Delta)$.
A tiger $E$ for $(X, \Delta)$ is called \textit{exceptional} if $E$ does not lie in $X$.
When $\Delta = \emptyset$, we say that $X$ has a tiger in $Y$ for short.
\end{defn}

\begin{lem}[{\cite[8.2.5 Definition-Lemma]{KM} and \cite[Lemma 3.12]{Lac}}]\label{lem:Sarkisov}
Let $(S, \Delta)$ be a log pair such that $S$ is a log del Pezzo surface of rank one.
Let $f \colon T \to S$ be an extraction of relative Picard rank one of an irreducible divisor $E$ of the minimal resolution of $S$.
The Mori cone $\overline{\mathrm{NE}}(T)$ has two edges, one of which is generated by $E$.
Let $R$ be the other edge.
Let $x=f(E)$, $\Gamma$ the log pullback of $\Delta$, and $\Gamma_\epsilon = \Gamma + \epsilon E$, where $0 < \epsilon \ll 1$.
Assume that $-(K_S+\Delta)$ is ample.
Then the following hold.
\begin{enumerate}
    \item[\textup{(1)}] $R$ is $K_T$-negative and contractible.
    Hence there is a rational curve $\Sigma$ that generate the same ray.
    Let $\pi$ be the contraction morphism of $R$.
    \item[\textup{(2)}] $K_T+\Gamma_\epsilon$ is anti-ample.
    \item[\textup{(3)}] $\Gamma_\epsilon$ is $E$-negative.
    \item[\textup{(4)}] There is a unique rational number $\lambda$ such that with $\Gamma' = \lambda\Gamma_\epsilon$, $K_T+\Gamma'$ is $R$-trivial. 
    Moreover, $\lambda >1$.
    \item[\textup{(5)}] $K_T+\Gamma'$ is $E$-negative.
    \item[\textup{(6)}] $\pi$ is either a $\PP^1_k$-fibration (called \textit{``net''}) or birational.
    \item[\textup{(7)}] If $\pi \colon T \to S_1$ is birational, and $\Delta_1 = \pi_*\Gamma'$, then $K_{S_1}+\Delta_1$ is anti-ample and $S_1$ is a log del Pezzo surface of rank one.
\end{enumerate}
\end{lem}

\begin{defn}[{\cite[8.2.8 Definition-Remark]{KM} and \cite[Definition 3.13]{Lac}}]
We call the transformation $(f, \pi)$ as in Lemma \ref{lem:Sarkisov} \textit{a hunt step} for $(S, \Delta)$ if $e(E; S, \Delta)$ is maximal among exceptional divisors of the minimal resolution of $S$.
If $x$ is a chain singularity (resp.\ a non-chain singularity) and $\Delta = \emptyset$, we require $E$ to not be a $(-2)$-curve (resp.\ to be the central curve). 
This is always possible by \cite[8.3.9 Lemma and 10.11 Lemma]{KM}.
\end{defn}

Hunt steps often preserve the flush condition. For example, we have:

\begin{lem}[{\cite[8.4.5 Lemma]{KM}}]\label{lem:flush}
We follow the notation of Lemma \ref{lem:Sarkisov}. 
In each hunt step with $\Delta = \emptyset$, if $\lfloor \Gamma' \rfloor=0$, then $K_T+\Gamma'$ is plt.
If $T$ is not a net in addition, then $K_{S_1}+\Delta_1$ is flush.
\end{lem}

On the other hand, the flush condition controls the singularity of pairs as follows.

\begin{lem}[{\cite[8.0.4 Lemma]{KM}}]\label{lem:singflush}
Let $(S, \Delta)$ be a flush log pair such that $S$ is a surface.
Let $A=\lceil \Delta \rceil$.
Then $(S, A)$ is plt at each singular point of $S$.
\end{lem}

\subsection{Du Val del Pezzo surface with Dynkin type $2[2^4]$}\label{subsec:2A4}

In this subsection, we compile facts on Du Val del Pezzo surfaces with Dynkin type $2[2^4]$ in characteristic $p \geq 0$.
Note that such a surface is unique up to isomorphism in each characteristic by \cite{Ye} and \cite{KN2}.

\begin{eg}\label{eg:2A4}
Let $k$ be an algebraically closed field of characteristic $p \geq 0$.
Let us recall the construction of the Du Val del Pezzo surface with Dynkin type $2[2^4]$ as in \cite[\S 5.1]{ABL} (see also \cite{Bea} and \cite[Lemma B.12]{Lac}).

Consider the following four points in $\PP^2_k$
\begin{align*}
    a =[-1:1:1], b=[-1:-1:1], c=[1:-1:1] \text{ and } d=[1:1:1].
\end{align*}
Let $L_{ab}$, $L_{ac}$, $L_{ad}$, $L_{bc}$, $L_{bd}$, and $L_{cd}$ be the six lines passing through two of them.
Consider the cubic curves $C_0=L_{ad}+L_{ac}+L_{bc}$ and $C_\infty = L_{ab}+L_{bd}+L_{cd}$.
Then the base locus of the pencil $C_t = C_0+tC_\infty$ consists of points $a, b, c, d$ with multiplicity two, and a point $[0:0:1]$ with multiplicity one.
Take $U_1$ as the blow-up of $\PP^2_k$ at $a, b, c,$ and $d$.
Let $E_a, \ldots, E_d$ be the exceptional divisor over $a, \ldots, d$ respectively.
Next take $U_2$ as the blow-up of $U_1$ at $E_a \cap L_{ab}, E_b \cap L_{bc}, E_c \cap L_{cd}$, and $E_d \cap L_{ad}$.
Let $F_a, \ldots, F_d$ be the exceptional divisor over $a, \ldots, d$ respectively.
Then there are two chains of four $(-2)$-curves which consist of the strict transforms of $E_a$, $L_{ad}$, $L_{bc}$, $E_c$ and $E_d$, $L_{cd}$, $L_{ab}$, $E_b$ respectively.
Hence $U_2$ is the minimal resolution of the Du Val del Pezzo surface with Dynkin type $2[2^4]$.

Fixing coordinates $[X:Y:Z]$ of $\PP^2_k$, we can write the equation of $C_t$ as 
\begin{align*}
    (Y^2-Z^2)(X+Y)+t(X^2-Z^2)(Y-X)=0.
\end{align*}
The singular member locus of the pencil $C_t$ in $\PP^1_{k, [s:t]}$ is the same as that of $|-K_{U_2}|$, and equals $\{st(t^2+11st-s^2)=0\}$.
The equation $t^2+11t-1=0$ has a double root if and only if $p=5$.
In $p \neq 5$ (resp.\ $p=5$), for $\alpha \in k$ with $\alpha^2+11\alpha-1=0$, the member $C_\alpha$ is a nodal cubic (resp.\ cuspidal cubic).

In $p=5$, by \cite[Theorem 4.1]{Lang2}, we can also construct the Du Val del Pezzo surface with Dynkin type $2[2^4]$ as the hypersurface $\{y^2-(x^3+2t^4x+4s^5t+2t^6)=0\}$ in $\PP(1,1,2,3)$, where $s$, $t$, $x$, and $y$ are coordinates of weight $1$, $1$, $2$, and $3$ respectively.
\end{eg}

\section{Pathological phenomena}\label{sec:implication}

In this section, we give implications between pathological phenomena in Definition \ref{notation}.
First let us recall the definition of log liftability.

\begin{defn}[{\cite[Definition 2.2]{KN1}}]\label{d-liftable}
Let $\alpha \colon S \to T$ be a morphism between Noetherian irreducible schemes.
Let $Y$ be a smooth projective scheme over $S$ and $E$ a simple normal crossing divisor over $S$ on $Y$ with $E = \sum_{i=1}^r E_i$ the irreducible decomposition. 
We say that the pair $(Y,E)$ \textit{lifts to $T$ via $\alpha$} if 
there exist 
\begin{itemize}
	\item a smooth and projective morphism $\mathcal{Y} \to T$ and
	\item effective divisors $\mathcal{E}_1, \dots, \mathcal E_r$ on $\mathcal{Y}$ such that $\sum _{i=1}^r \mathcal{E}_i$ is simple normal crossing over $T$
\end{itemize}
such that the base change of the schemes $\mathcal{Y}, \mathcal{E}_1,\cdots,\mathcal{E}_r$ by $\alpha \colon S \to T$ are isomorphic to $Y, E_1,\cdots, E_r$ respectively. 
When $T$ is the spectrum of a local ring $(R, m)$ and $\alpha$ is induced by $R/m \cong k$, we also say that $(Y, E)$ \textit{lifts to $R$} for short.
\end{defn}

\begin{defn}[{\cite[Definition 2.5]{KN1}}]
Let $X$ be a normal projective surface. Fix a Noetherian irreducible scheme $T$ and a morphism $\alpha \colon \Spec k \to T$.
We say that $X$ is \textit{log liftable over $T$ via $\alpha$} (or \textit{log liftable over $R$ via $\alpha$} when $T= \Spec R$) if the pair $(Z, E_{f})$ lifts to $T$ via $\alpha$ for some log resolution $f \colon Z \to X$.
When $T$ is the spectrum of a local ring $(R, m)$ and $\alpha$ is induced by $R/m \cong k$, 
we also say that $X$ is \textit{log liftable over $R$} for short. 
\end{defn}

By the following lemma, we only have to consider the minimal resolution to check the property (NL).

\begin{lem}[{\cite[Lemma 2.6]{KN1}}]
Let $X$ be a normal projective surface with only rational singularities.
Then the following are equivalent.
\begin{enumerate}
    \item[\textup{(1)}] For every resolution $f \colon Z \to X$, the pair $(Z, E_f)$ lifts to $W(k)$.
    \item[\textup{(2)}] For some resolution $f \colon Z \to X$, the pair $(Z, E_f)$ lifts to $W(k)$.
    \item[\textup{(3)}] For the minimal resolution $\pi \colon Y \to X$, the pair $(Y, E_\pi)$ lifts to $W(k)$.
\end{enumerate}
\end{lem}

Next, we show Theorem \ref{thm:main2} (1), i.e., (ND) $\Rightarrow$ (NL) when the rank is one.

\begin{lem}\label{lem:NDtoNL}
Let $S$ be a log del Pezzo surface of rank one.
Let $R$ be a Noetherian integral domain of characteristic zero with a morphism $\alpha \colon \Spec{k} \to \Spec{R}$. 
If $X$ is log liftable over $R$ via $\alpha$, then there exists a log del Pezzo surface of rank one over $\C$ which has the same Dynkin type as $S$.
\end{lem}

\begin{proof}
Let $\sigma \colon \wt S \to S$ be the minimal resolution and $(\wt{\mc{S}}, \mc E)$ be an $R$-lifting of $(\wt S, E_\sigma)$.
Shrinking the field of definition of the generic fiber of $(\wt{\mc{S}}, \mc E)$ if necessary, we obtain the base change $(\wt{S}_\C, E_\C)$ of the generic fiber to $\Spec \C$.
Since $E_\C$ have the same intersection matrix as $E_\sigma$, we have a contraction $\sigma_\C \colon \wt{S}_\C \to S_\C$ of $E_\C$ by \cite[Theorem 3.9]{Bad}.
Since $H^1(\wt{S}_\C, \sO_{\wt{S}_\C}) \cong H^1(\wt{S}, \sO_{\wt{S}}) \cong 0$ and $H^0(\wt{S}_\C, 2K_{\wt{S}_\C}) \cong H^0(\wt{S}, 2K_{\wt{S}}) \cong 0$, the Castelnuovo's criterion shows that $\wt{S}_\C$ is a rational surface.
Since $K_{\wt{S}_\C}^2=K_{\wt{S}}^2$, the Picard rank of $\wt{S}_\C$ coincides with that of $\wt{S}$, and hence the Picard rank of $S_\C$ is one.
On the other hand, take $\mc H$ as a Cartier divisor on $\wt{\mc{S}}$ relatively ample over $R$ and write $H_k$ (resp.\  $H_\C$) as the restriction of $\mc H$ to $\wt S$ (resp.\ $\wt{S}_\C$).
Then 
\begin{align*}
    0>(K_S \cdot \sigma_* H_k)=(\sigma^* K_{\wt S} \cdot H_k)=(\sigma^*_\C K_{\wt{S}_\C} \cdot H_\C)=(K_{S_\C} \cdot (\sigma_\C)_* H_\C),
\end{align*}
which implies that $-K_{S_\C}$ is ample.
Hence $S_\C$ is the desired surface.
\end{proof} 

Next, we show (NK) $\Rightarrow$ (NB) for log del Pezzo surfaces of rank one with sufficiently small anti-canonical volume. 
The proof is similar to that of \cite[Theorem 4.8]{Kaw}.

\begin{lem}\label{lem:NKtoNB}
Let $S$ be a log del Pezzo surface of rank one with index $r$ in characteristic $p>0$ and $A$ an ample $\Z$-Weil divisor on $S$.
Suppose that a general member of $|-rK_S|$ is smooth.
Then $H^1(S, \sO_S (-A)) = 0$ unless $p(-K_S \cdot A) \leq (r-1) K_S^2$.
In particular, it holds that $H^1(S, \sO_S (-A)) = 0$ if $p > r(r-1)K_S^2$.
\end{lem}

\begin{proof}
Suppose that $H^1(S, \sO_S (-A)) \neq 0$.
Then $H^0(S, (\Omega_S^{[1]} \otimes \sO_S(-p^e A))^{**}) \neq 0$ for some $e \in \Z_{>0}$ by \cite[Lemma 2.5]{Kaw}, where $\Omega_S^{[1]}$ is the reflexive hull of $\Omega_S^1$.
In other words, we have an injection $s \colon \sO_S(p^e A) \hookrightarrow \Omega_S^{[1]}$.
Note that $|-rK_S|$ has no fixed component because its general member is smooth and connected and the Riemann-Roch theorem shows that $\dim |-rK_S| \geq (r+1)/2 \geq 1$.
In particular, we can choose a smooth member $C \in |-rK_S|$ so that $s|_C \colon \sO_C(p^e A) \rightarrow \Omega_S^{[1]}|_C$ is still injective and $\Omega_S^{[1]}|_C=\Omega_S^1|_C$.
If the composition $t \colon \sO_S(p^e A) \to \omega_C$ of $s|_C$ and the canonical map $\Omega_S^1|_C \to \omega_C$ is zero, then the conormal bundle sequence with respect to $C \subset S$ gives an injection $\sO_C(p^e A) \hookrightarrow \sO_C(-C)$, a contradiction since $(p^e A \cdot C)>0> -C^2$.
Hence $t$ is nonzero, which implies that $\deg \sO_C(p^e A) \leq \deg \omega_C$.
Hence $p^e(-rK_S \cdot A) \leq (r-1) rK_S^2$, and the assertion holds.
\end{proof}

Finally, let us see that $p=5$ is the largest positive characteristic in which log del Pezzo surfaces of rank one can satisfy (NL), (NK), or (ND). 

\begin{lem}
Suppose that $p >5$.
Then no log del Pezzo surfaces of rank one satisfy (NL), (NK), or (ND).
\end{lem}

\begin{proof}
By \cite[Theorem 7.2]{Lac}, such surfaces lifts to characteristic zero over a smooth base.
\cite[Proposition 2.5]{ABL} now shows that they do not satisfy (NL).
Hence the assertion follows from \cite[Proposition 3.4]{KN1} and Lemma \ref{lem:NDtoNL}.
\end{proof}

\section{Log del Pezzo surfaces without tigers in characteristic five}\label{sec:notiger}

In this section, we construct three log del Pezzo surfaces in characteristic $p=5$.
Only those were erroneously omitted in \cite[\S 6.1]{Lac}
among non-log liftable log del Pezzo surfaces in $p=5$ which may have no tigers when we submitted this paper.
To construct them, let us consider a certain configuration of curves in $\PP^2_k$.

\begin{lem}\label{lem:specialconfig}
In $\PP^2_k$, let $C$ be a cuspidal curve and $Q$ a smooth conic intersecting with $C$ at a smooth point $t$ of $C$ with multiplicity at least five.
Take a point $s$ so that $C \cap Q=\{s, t\}$.
(We admit that $s=t$.)
Let $u$ be the cusp of $C$.
Take the line $M_u$ which intersects with $C$ at $u$ with multiplicity three and the line $L_{su}$ (resp.\ $L_{tu}$) passing through $s$ (resp.\ $t$ ) and $u$.
Then the following hold:
\begin{enumerate}
    \item[\textup{(1)}] When $p \neq 3$, then there are coordinates $[x:y:z]$ of $\PP^2$ such that $C=\{x^3=y^2z\}$, $Q = \{-45x^2-5y^2+z^2+24xy+40yz-15zx=0\}$, $M_u=\{y=0\}$, and $t=[1:1:1]$.
    \item[\textup{(2)}] When $p = 3$, then there are coordinates $[x:y:z]$ of $\PP^2$ such that $C = \{x^3=y^2z+x^2y\}$, $Q = \{-x^2+z^2-yz-zx=0\}$, $M_u=\{y=0\}$, and $t=[0:1:0]$.
    \item[\textup{(3)}] $Q \cap L_{tu}$ consists of exactly two points.
    \item[\textup{(4)}] $s=t$ if and only if $p=2$.
    \item[\textup{(5)}] $M_u$ is the tangent line of $Q$ at some point if and only if $p=5$.
    \item[\textup{(6)}] $L_{su}$ is the tangent line of $Q$ at $s$ if and only if $p=5$.
\end{enumerate}
\end{lem}

\begin{proof}
\noindent (1): 
By \cite[Theorem 3.1]{LPS}, we can choose the coordinates $[x:y:z]$ of $\PP^2_k$ such that $C = \{x^3=y^2z\}$ and $M_u=\{y=0\}$.
Since the automorphism $[x:y:z] \mapsto [\alpha x:y:\alpha^3 z]$ of $\PP^2_k$ with $\alpha \in k^*$ fixes $C$, we may assume that $t=[0:1:0]$ or $[1:1:1]$.
Since $Q$ does not pass through $u = [0:0:1]$, we can write $Q = \{ax^2+by^2+z^2+dxy+eyz+fzx=0\}$ for some $a, b,d,e,f \in k$.
Since 
\begin{align*}
    C \cap Q \cap \{y \neq 0\} &\cong \{x^3=z, ax^2+b+z^2+dx+ez+fzx=0\} \subset \AA^2_{k, (x,z)}\\
    & \cong \{x^6+fx^4+ex^3+ax^2+dx+b=0\} \subset \AA^1_{k, (x)}, 
\end{align*}
either $0$ or $1$ is the root of multiplicity at least five of $x^6+fx^4+ex^3+ax^2+dx+b=0$.
If the former holds, then $a=b=d=e=f=0$ and hence $Q=\{z^2=0\}$, a contradiction.
Hence the latter holds.
Since the coefficient of $x^5$ is zero, we obtain $x^6+fx^4+ex^3+ax^2+dx+b=(x-1)^5(x+5)$.
Comparing coefficients, we obtain the assertion.

\noindent (2):
As in the proof of the assertion (1), if $C = \{x^3=y^2z\}$, then $Q=\{-45x^2-5y^2+z^2+24xy+40yz-15zx=0\}=\{(y+z)^2=0\}$, a contradiction.
Hence we can choose the coordinates $[x:y:z]$ of $\PP^2_k$ such that $C = \{x^3=y^2z+x^2y\}$ and $M_u=\{y=0\}$ by \cite[Theorem 3.1]{LPS}.
Since the automorphism $[x:y:z] \mapsto [x+\alpha y:y:\alpha x+(\alpha^3-\alpha^2)y+z]$ of $\PP^2_k$ with $\alpha \in k$ fixes $C$, we may assume that $t=[0:1:0]$.
Since $Q$ does not pass through $u=[0:0:1]$, we can write $Q = \{ax^2+by^2+z^2+dxy+eyz+fzx=0\}$ for some $a, b,d,e,f \in k$.
Since 
\begin{align*}
    C \cap Q \cap \{y \neq 0\} &\cong \{x^3-x^2=z, ax^2+b+z^2+dx+ez+fzx=0\} \subset \AA^2_{k, (x,z)}\\
    & \cong \left\{
    \begin{array}{l}
    x^6+x^5+(1+f)x^4+(e-f)x^3\\
    +(a-e)x^2+dx+b=0
    \end{array}
    \right\} \subset \AA^1_{k, (x)}, 
\end{align*}
$0$ is the root of multiplicity five of $x^6+x^5+(1+f)x^4+(e-f)x^3+(a-e)x^2+dx+b=0$.
Comparing coefficients, we obtain the assertion.

\noindent (3):
In $p \neq 3$ (resp.\ $p=3$), we have $L_{tu}=\{x=y\}$ (resp.\ $\{x=0\}$) and $L_{tu} \cap Q=\{[1:1:1], [1:1:-26]\}$ (resp.\ $\{[0:1:0], [0:1:1]\}$).
Hence the assertion holds.

\noindent (4): 
In $p \neq 3$ (resp.\ $p=3$), we have $C \cap Q=\{[1:1:1], [-5:1:-125]\}$ (resp.\ $\{[0:1:0], [-1:1:1]\}$).
Hence the assertion holds.

\noindent (5): 
In $p \neq 3$ (resp.\ $p=3$), $Q \cap M_u$ is isomorphic to $\{z^2-15z-45=0\} \subset \AA^1_{(z)}$ (resp.\ $\{z^2-z-1=0\} \subset \AA^1_{(z)}$), whose discriminant equals $405=5 \cdot 3^4$ (resp.\ $5$).
Hence the assertion holds.

\noindent (6):
Suppose that $L_{su}$ is the tangent line of $B$ at $s$
Assume that $p = 3$.
Then $s=[-1:1:1]$ and $L_{su}=\{x=-y\}$.
However, $Q \cap L_{su} \cong \{x=-y, (x-z)(x+z)=0\}$ consists of two points, a contradiction.

Hence $p \neq 3$, $s=[-5:1:-125]$, and $L_{su}=\{x=-5y\}$.
Since 
\begin{align*}
    Q \cap L_{su} \cap \{y \neq 0\} &\cong \{x=-5, -45x^2-5+z^2+24x+40z-15zx=0\} \subset \AA^2_{k, (x,z)}\\
    & \cong \{(z+125)(z-10)=0\} \subset \AA^1_{k, (z)}
\end{align*}
consists of one point, it holds that $125+10 = 135=5 \cdot 3^3$ equals zero in $k$. 
Hence the assertion holds.
\end{proof}

By Lemma \ref{lem:specialconfig} (5) and (6), we can construct pathological log del Pezzo surfaces in $p=5$ as follows.

\begin{eg}\label{eg:Lac4.8cex-1}
In $p = 5$, we follow the notation as in Lemma \ref{lem:specialconfig}.
Set $(n_t, n_u)=(5,3)$, $(6, 3)$, or $(5, 4)$.
Next blow-up at $s$ twice along along $L_{su}$.
Next blow-up at $t$ $n_t$ times along $C$.
Finally blow-up at $u$ $n_u$ times along $C$.

When $(n_t, n_u)=(5,3)$, $(6, 3)$, and $(5, 4)$, we obtain the minimal resolution $\wt S$ of a log del Pezzo surface $S$ of rank one with $\Dyn(S)=2[3,2]+[3]+[2]+[2^4]$, $[4,2]+[3,2^5]+[3,2]+[2]$, and $[2,4]+[2,3,2^2]+[3]+[2^4]$ respectively.
Indeed, the strict transform of the $(-1)$-curve over $u$ in $S$ is of $(-K_S)$-degree $\frac15$, $\frac{1}{35}$, or $\frac{5}{77}$, and hence $-K_S$ is ample.
Since the Dynkin types violate the Bogomolov bound \cite[9.2 Corollary]{KM}, these surfaces satisfy (ND) and (NL) by Lemma \ref{lem:NDtoNL}.
\end{eg}

\section{Extractions of tigers}\label{sec:Lac}

In this section, we prove a refinement of \cite[Theorem 6.2]{Lac} and an alternative of \cite[Theorem 6.25]{Lac}.
To this end, we recall the definition of almost log canonical singularities.

\begin{defn}[{\cite[Definition 3.6]{Lac}}]\label{def:almostlc}
Let $s \in S$ be the germ of a klt surface singularity and $D$ a reduced $\Z$-Weil divisor on $S$.
We say that a pair $(S, D)$ is \textit{almost log canonical} if the coefficient of every exceptional divisor of the minimal resolution $\sigma \colon \wt S \to S$ with respect to $(S, D)$ is at most one.
When $s \in S$ is singular, $(S, D)$ is almost log canonical if and only if either $(S, D)$ is log canonical at $s$ or, by replacing $S$ as a suitable \'{e}tale neighborhood of $s$ if necessary, the $\sigma$-exceptional divisor $E_\sigma$ satisfies one of the following.
\begin{enumerate}
    \item[\textup{(a)}] $E_\sigma$ is irreducible and the strict transform of $D$ in $\wt S$ is simply tangent to $E_\sigma$.
    \item[\textup{(b)}] $E_\sigma$ is irreducible and the strict transform of $D$ in $\wt S$ has two branch meeting transversally on $E_\sigma$, with each branch meeting $E_\sigma$ transversally as well.
    \item[\textup{(c)}] $E_\sigma$ consists of two irreducible components $E_1$ and $E_2$, and the strict transform of $D$ in $\wt S$ has one branch meeting each component transversally at $E_1 \cap E_2$.
\end{enumerate}
\end{defn}

Next, we prepare an auxiliary lemma. 

\begin{lem}\label{LacA.19}
Let $S$ be a projective surface with klt singularities and $\sigma \colon \wt S \to S$ be the minimal resolution.
Take a reduced and irreducible curve $C \subset S$ such that $(K_S+C \cdot C) \leq 0$.
Then one of the following holds.
\begin{enumerate}
    \item[\textup{(1)}] $C \subset S^\circ$.
    \item[\textup{(2)}] For each $s \in S$, $\sigma$ is a log resolution of $(S, C)$ at $s$ or $E_\sigma$ satisfies one of (a)--(c) in Definition \ref{def:almostlc} at $s$.
\end{enumerate}
Moreover, if $C$ has a singularity in $S^\circ$ in addition, then $C \subset S^\circ$.
\end{lem}

\begin{proof}
Let $E_\sigma \coloneqq \sum_{i=1}^n E_i$ be the irreducible decomposition and $\wt C \coloneqq \sigma_*^{-1} C$.
Take non-negative rational vectors $\bm{a} = (a_i)_{i=1}^n, \bm{b} = (b_i)_{i=1}^n, \bm{d} = (d_i)_{i=1}^n, \bm{e} = (e_i)_{i=1}^n$ and $\bm{f} = (f_i)_{i=1}^n$ so that $a_i= (\wt C \cdot E_i)$, $\sigma^* C = \wt C + \sum_{i=1}^n d_i E_i$, $\sigma^* K_S =K_{\wt S} +\sum_{i=1}^n e_i E_i$, $\bm{b}=\bm{d}+\bm{e}$, and $\bm{f}=\bm{1} - \bm{b}$, where $\bm{1}=(1)_{i=1}^n$.
We will use the symbol $\langle$-,-$\rangle$ to denote the inner product of vectors.
Write $M =((E_i \cdot E_j))_{i=1, j=1}^{n, n}$ for the intersection matrix of $E_\sigma$.
By assumption, we have 
\begin{align*}
0 &\geq (K_S +C \cdot C) = (K_{\wt S}+ \wt C \cdot \sigma^* C) = (K_{\wt S}+\wt C \cdot \wt C) + \sum_{i=1}^n d_i (K_{\wt S} \cdot E_i) + \langle \bm{d}, \bm{a} \rangle \\
&\geq (K_{\wt S}+\wt C \cdot \wt C) = 2 (p_a(\wt C)-1).
\end{align*}

Suppose that $p_a(\wt C) \geq 1$.
Then we have $p_a(\wt C)=1$.
Moreover $d_i=0$ or $a_i=0$ for each $1 \leq i \leq n$, and hence $C \subset S^\circ$.
In particular, the last assertion holds.

In what follows, we may assume that $p_a( \wt C)=0$ and $C \not\subset S^\circ$.
Taking a resolution of $S \setminus C$ if necessary, we may also assume that $\Sing S \subset C$.
Then 
\begin{align*}
(K_S +C \cdot C) &= (\sigma^*(K_S + C) \cdot \wt C) = (K_{\wt S}+\wt C \cdot \wt C) + \langle \bm{b}, \bm{a} \rangle \\
&= -2+ \langle \bm{b}, \bm{a} \rangle  = -2+\sum_{i=1}^n a_i - \langle \bm{f}, \bm{a}\rangle
\end{align*}
and hence
\begin{align*}
0 \geq (K_S+C \cdot C) \iff \langle \bm{a}, \bm{b} \rangle \leq 2 \iff \langle \bm{a}, \bm{f}\rangle \geq \sum_{i=1}^n a_i -2. 
\end{align*}

Set $\Sing S = \{x_1, \ldots, x_m\}$ and take the decomposition $\bm{a}=\sum_{j=1}^m \bm{a}_j$ and $\bm{b} = \sum_{j=1}^m \bm{b}_j$ with respect to $\{x_1, \ldots, x_m\}$ so that $\langle \bm{a}, \bm{b} \rangle= \sum_{j=1}^m \langle \bm{a}_j, \bm{b}_j \rangle$.
If $(K_S+C \cdot C) \leq 0$, then $\langle \bm{a}_j, \bm{b}_j \rangle \leq 2$ for all $1 \leq j \leq m$.
Since the condition (2) is local, we may assume that $\Sing S$ consists of one point contained in $C$, and 
\begin{align}
    \langle \bm{a}, \bm{f} \rangle \geq \sum_{i=1}^n a_i -2. \label{eq:lemA1}
\end{align}

On the other hand, $M^{-1}$ has only negative entries by \cite[2.19.3 Lemma]{Kol92}.
If $\bm{a}$ becomes (strictly) smaller, then so does $\langle \bm{a}, \bm{b} \rangle$ since $\bm{d}=-M^{-1} \bm{a}$ and $\bm{b}=\bm{d}+\bm{e}$.
In particular, if (\ref{eq:lemA1}) holds for some $\bm{a}$, then it also holds, but the equality does not hold, for smaller $\bm{a}$.

Now let us show the first assertion by induction on $\bm{a}$.
Set $\Gamma$ as the weighted dual graph of $E_\sigma$ with the vertex $v_i$ corresponding to $E_i$.
Without further reference, we will use \cite[(3.1.10)]{Kol92} to calculate the log discrepancies $\bm{f}$.

\textbf{Case 1}:
Suppose that $\Gamma$ is a chain.

\textbf{Case 1a}:
Suppose further that $a_i \neq 0$ if and only if $i=1$.
The subgraph $\Gamma - v_1= \Gamma_1+\Gamma_2$ is the disjoint union of two chains (each $\Gamma_i$ could be empty). Then
\begin{align*}
    f_1 =\frac{\Delta(\Gamma_1)\Delta(\Gamma_2)}{\Delta(\Gamma)}(\frac1{\Delta(\Gamma_1)}+\frac1{\Delta(\Gamma_2)}-a_1).
\end{align*}
If $a_1=1$, then $\sigma$ is a log resolution of $(S, C)$.
Suppose further that $a_1 \geq 2$.
Then $f_1 \geq 0$ since $a_1 f_1 = \langle \bm{a}, \bm{f} \rangle \geq \sum_{i=1}^n a_i -2 \geq 0$.
Hence $\Delta(\Gamma_1)=\Delta(\Gamma_2)=1$ and $a_1=2$.
Hence $E_\sigma$ is irreducible.
Moreover, either $\sigma$ is a log resolution of $(S, C)$ or $E_\sigma$ satisfies (a) or (b) of Definition \ref{def:almostlc}.
Note that when $a_1=2$, we have $\langle \bm{a}, \bm{f} \rangle= \sum_{i=1}^n a_i -2$ and hence (\ref{eq:lemA1}) does not hold for larger $\bm{a}$.

\textbf{Case 1b}:
Suppose further that $a_i \neq 0$ if and only if $i \leq 2$.
The subgraph $\Gamma - v_1-v_2= \Gamma_1+\Gamma_2+\Gamma_3$ is the disjoint union of three chains (each $\Gamma_i$ could be empty) such that $\Gamma_{1,2} \coloneqq \Gamma_1+v_1+\Gamma_2$ and $\Gamma_{2,3} \coloneqq \Gamma_2+v_2+\Gamma_3$ are connected.
Since (\ref{eq:lemA1}) also holds for $\bm{a}-(a_1, 0^{n-1})$ and $\bm{a}-(0, a_2, 0^{n-2})$, the consequence of Case 1a shows that $a_1 =a_2=1$.
Hence (\ref{eq:lemA1}) for $\bm{a}$ is equivalent to $f_1+f_2 \geq 0$.
Hence we may assume that $f_1 \geq 0$.
On the other hand, we have
\begin{align*}
   f_1 = \frac{\Delta (\Gamma_1) \Delta(\Gamma_{2,3})}{\Delta(\Gamma)}\left( \frac{1-\Delta(\Gamma_1)}{\Delta(\Gamma_1)}+\frac{1-\Delta(\Gamma_3)}{\Delta(\Gamma_{2,3})}\right) \geq 0.
\end{align*}
Hence $\Delta(\Gamma_1)=\Delta(\Gamma_3)=1$.
In particular, $\Gamma_1=\Gamma_3=\emptyset$.
Hence either $\sigma$ is a log resolution of $(S, C)$ or $E_\sigma$ satisfies (c) of Definition \ref{def:almostlc}.
Since it holds that $\langle \bm{a}, \bm{f} \rangle= \sum_{i=1}^n a_i -2$, (\ref{eq:lemA1}) does not hold for larger $\bm{a}$.

\textbf{Case 1c}:
Suppose further that $a_{i} \neq 0$ if $i \leq 3$.
Then (\ref{eq:lemA1}) also holds for $(a_1, a_2, 0^{n-2})$, a contradiction with the consequence of Case 1b.

\textbf{Case 2}: 
Suppose that $\Gamma$ is not a chain.
Let $v_n$ be the center of $\Gamma$ and set $\Gamma_1, \Gamma_2$ and $\Gamma_3$ be the branches of $\Gamma$.

\textbf{Case 2a}:
Suppose further that $a_i \neq 0$ if and only if $i=n$.
Then 
\begin{align*}
    f_n = \frac{\Delta(\Gamma_1) \Delta(\Gamma_2) \Delta(\Gamma_3)}{\Delta(\Gamma)} \left( \frac1{\Delta(\Gamma_1)} + \frac1{\Delta(\Gamma_2)} + \frac1{\Delta(\Gamma_3)} - (1+a_n)\right).
\end{align*}
If $a_n \geq 2$, then (\ref{eq:lemA1}) shows that 
\begin{align*}
    \frac32 \geq \frac1{\Delta(\Gamma_1)} + \frac1{\Delta(\Gamma_2)} + \frac1{\Delta(\Gamma_3)} \geq 1+a_n \geq 3,
\end{align*}
a contradiction. 
Hence $a_1=1$ and $\sigma$ is a log resolution of $(S, C)$.

\textbf{Case 2b}:
Suppose further that $v_1 \in \Gamma_1$ and $a_i \neq 0$ if and only if $i=1$.
The subgraph $\Gamma_1 - v_1=\Gamma_{1,1}+\Gamma_{1,2}$ is the disjoint union of two chains (each $\Gamma_{1,i}$ could be empty) such that $\Gamma_{1,2}$ is adjacent to $v_n$.
Let $\Gamma' \coloneqq \Gamma -\Gamma_{1,1} - v_1$.
Then
\begin{align*}
    f_1 = \frac{\Delta(\Gamma_{1,1}) \Delta(\Gamma)}{\Delta(\Gamma)} \left( \frac{1-(\Delta(\Gamma_3)-1)(\Delta(\Gamma_2)-1)}{\Delta(\Gamma')} + \frac{1}{\Delta(\Gamma_{1,1})} - a_1 \right).
\end{align*}
If $a_1 \geq 2$, then (\ref{eq:lemA1}) shows that
\begin{align*}
    1 \geq \frac{1-(\Delta(\Gamma_3)-1)(\Delta(\Gamma_2)-1)}{\Delta(\Gamma')} + \frac{1}{\Delta(\Gamma_{1,1})} \geq a_1 \geq 2,
\end{align*}
a contradiction.
Hence $a_1 = 1$ and $\sigma$ is a log resolution of $(S, C)$.

\textbf{Case 2c}:
Suppose further that $v_1 \in \Gamma_1$, $v_2 \in \Gamma_2$ and $a_i \neq 0$ if and only if $i \leq 2$.
Then for $i=1, 2$, $\Gamma_i - v_i=\Gamma_{i,1}+\Gamma_{i,2}$ is the disjoint union of two chains (each $\Gamma_{i,j}$ could be empty) such that $\Gamma_{i,2}$ is adjacent to $v_n$.
Set $\Gamma' \coloneqq \Gamma - \Gamma_{1,1} - v_1$.
Since (\ref{eq:lemA1}) holds for $\bm{a}-(a_1, 0^{n-1})$ and $\bm{a}-(0, a_2, 0^{n-2})$, the consequence of Case 2b show that $a_1 =a_2=1$.
Hence (\ref{eq:lemA1}) for $\bm{a}$ is equivalent to $f_1+f_2 \geq 0$.
Hence we may assume that $f_1 \geq 0$.
However, we have
\begin{align*}
    f_1 &= \frac{\Delta(\Gamma') \Delta(\Gamma_{1,1})}{\Delta(\Gamma)} \left( \frac{1 - (\Delta(\Gamma_2) -1) (\Delta(\Gamma_3) - 1)}{\Delta(\Gamma')} + \frac{1-\Delta(\Gamma_{1,1})}{\Delta(\Gamma_{1,1})} - \frac{\Delta(\Gamma_{2,1}) \Delta(\Gamma_3)}{\Delta(\Gamma')}\right)\\
    & < 0,
\end{align*}
a contradiction.

\textbf{Case 2d}:
Suppose further that $v_1, v_2 \in \Gamma_1$ and $a_i \neq 0$ if and only if $i \leq 2$.
Then $\Gamma_1 -v_1 - v_2 = \Gamma_{1,1}+\Gamma_{1,2} + \Gamma_{1,3}$ is the disjoint union of three chains (each $\Gamma_{1,i}$ could be empty). 
We may assume that $\Gamma_{1,3}$ is adjacent to both $v_2$ and $v_n$, and $\Gamma_{1,2}$ is adjacent to both $v_1$ and $v_2$.
Set $\Gamma_a \coloneqq \Gamma - \Gamma_{1,1} - v_1$. $\Gamma_b \coloneqq \Gamma_a - \Gamma_{1,2} - v_2$, and $\Gamma_c \coloneqq \Gamma - \Gamma_b - v_2$.
Since (\ref{eq:lemA1}) holds for $\bm{a}-(a_1, 0^{n-1})$ and $\bm{a}-(0, a_2, 0^{n-2})$, the consequence of Case 2b show that $a_1 =a_2=1$.
Hence (\ref{eq:lemA1}) for $\bm{a}$ is equivalent to $f_1+f_2 \geq 0$.
However, we have
\begin{align*}
    f_1 &= \frac{\Delta (\Gamma_{1,1}) \Delta (\Gamma_a)}{\Delta (\Gamma)} \left( \frac{1 - (\Delta (\Gamma_2) - 1) (\Delta (\Gamma _3) -1)}{\Delta (\Gamma_a)} + \frac{1- \Delta(\Gamma_{1,1})}{\Delta (\Gamma_{1,1})} - \frac{\Delta (\Gamma_b)}{\Delta (\Gamma_a)} \right) <0,\\
    f_2 &= \frac{\Delta (\Gamma_{b}) \Delta (\Gamma_c)}{\Delta (\Gamma)} \left( \frac{1 - (\Delta (\Gamma_2) - 1) (\Delta (\Gamma _3) -1)}{\Delta (\Gamma_b)} + \frac{1- \Delta(\Gamma_c)}{\Delta (\Gamma_{c})} - \frac{\Delta (\Gamma_{1,1})}{\Delta (\Gamma_c)} \right) <0,
\end{align*}
a contradiction.

\textbf{Case 2e}:
Suppose further that $v_1 \in \Gamma_1$ and $a_i \neq 0$ if and only if $i = 1$ or $n$.
The subgraph $\Gamma_1 - v_1=\Gamma_{1,1}+\Gamma_{1,2}$ is the disjoint union of two chains (each $\Gamma_{1,i}$ could be empty) such that $\Gamma_{1,2}$ is adjacent to $v_n$.
Set $\Gamma' \coloneqq \Gamma - \Gamma_{1,1} - v_1$.
Since (\ref{eq:lemA1}) holds for $\bm{a}-(a_1, 0^{n-1})$ and $\bm{a}-(0^{n-1}, a_n)$, the consequences of Cases 2a and 2b show that $a_1 =a_n=1$.
Hence (\ref{eq:lemA1}) for $\bm{a}$ is equivalent to $f_1+f_n \geq 0$.
However, we have
\begin{align*}
    f_1&=\frac{\Delta (\Gamma') \Delta (\Gamma_{1,1})}{\Delta (\Gamma)} \left( \frac{1 - (\Delta (\Gamma_2) - 1) (\Delta (\Gamma _3) -1)}{\Delta (\Gamma')} + \frac{1- \Delta(\Gamma_{1,1})}{\Delta (\Gamma_{1,1})} - \frac{\Delta (\Gamma_2) \Delta(\Gamma_3) }{\Delta (\Gamma')} \right)<0,\\
    f_n&=\frac{\Delta(\Gamma_1)\Delta(\Gamma_2)\Delta(\Gamma_3)}{\Delta(\Gamma)} \left( \frac{1}{\Delta(\Gamma_1)} + \frac{1}{\Delta(\Gamma_2)} + \frac{1}{\Delta(\Gamma_3)} -2 - \frac{\Delta(\Gamma_{1,1})}{\Delta(\Gamma_1)}\right) <0,
\end{align*}
a contradiction.

\textbf{Case 2f}:
Suppose further that there is an integer $l \geq 3$ such that $a_{1}, a_{2}, a_{l} \neq 0$.
Then (\ref{eq:lemA1}) holds for $(a_1, a_2, 0^{n-2})$, a contradiction with the consequences of Cases 2c and 2d.

Combining these results, we obtain the assertion.
\end{proof}

In Remark \ref{rem:1sthunt}, we will see that $(S, C)$ as in Lemma \ref{LacA.19} can have an almost log canonical singularity and the minimal resolution of $S$ is not a log resolution of $(S, C)$ in general.
Using Lemma \ref{LacA.19}, we can show that an analogue of a hunt step adopted in \cite[Theorem 6.2]{Lac} outputs a plt pair or a Du Val del Pezzo surface of rank one as follows.

\begin{thm}[{cf. \cite[Theorem 6.2]{Lac}}]\label{Lac6.2}
Let $S$ be a rank one log del Pezzo surface with tigers in its minimal resolution $\wt S$.
\begin{enumerate}
    \item[\textup{(1)}] Suppose that $S$ has an exceptional tiger $E$ in $\wt S$.
    Take $f$ and $\pi$ as in Lemma \ref{lem:Sarkisov}.
    Let $C$ be the support of $\Delta_1$ if $\pi$ is birational.
    Then, by replacing $E$ if necessary, one of the following holds.
    \begin{enumerate}
        \item[\textup{(a)}] $T$ is a net, $E$ is a section or a bisection, and $K_T+E$ is plt.
        \item[\textup{(b)}] $\pi$ is birational, $S_1$ is a Du Val del Pezzo surface of rank one, and $C$ is an irreducible, reduced and singular anti-canonical member of $S_1$ contained in $S_1^\circ$.
        The image of the $\pi$-exceptional curve in $S_1$ is the singularity of $C$.
        \item[\textup{(c)}] $\pi$ is birational and $S_1$ is a log del Pezzo surface of rank one. 
        $K_{S_1}+C$ is plt and anti-nef.
        The image of the $\pi$-exceptional curve in $S_1$ is contained in $C$.
    \end{enumerate}
    \item[\textup{(2)}] If $S$ has tigers in $\wt S$ but none of which is exceptional, then $S$ 
    contains a reduced and irreducible curve $C$ such that $K_S+C$ is plt and anti-nef. 
\end{enumerate}
\end{thm}

\begin{proof}
(1): 
We follow the notation of Lemma \ref{lem:Sarkisov}.
By definition, there is an effective $\Q$-divisor $\alpha$ such that $K_S+\alpha$ is anti-nef and 
$K_T+E+f^{-1}_*(\alpha) = f^*(K_S+\alpha)$.
We may assume that $E$ has maximal coefficient for $(S, \emptyset)$ among exceptional tigers in $\wt S$.

Let us show that we can replace $E$ so that $K_T+E$ is plt.
For this, we may assume that $x = f(E)$ is a non-chain singularity and $E$ is distinct from the central curve, say $E_0$.
Then there is a unique non-plt point of $K_T+E$, say $y$.
Let $\Gamma$ be the dual graph of the exceptional divisor in $\wt S$ over $y$ and $v$ (resp.\ $v_0$) the vertex of $\Gamma$ corresponding to $E$ (resp.\ $E_0$).
Then the subgraph $\Gamma - v_0 =\Gamma_1+\Gamma_2+\Gamma_3$ is the disjoint union of three chains ($\Gamma_1$ could be empty) such that $\Gamma_1$ is adjacent to $v$ and $v_0$.
By \cite[(3,1,10)]{Kol92}, the log discrepancy of $E_0$ with respect to $K_T+E$ is
\begin{align*}
    \frac{\Delta(\Gamma_1)\Delta(\Gamma_2)\Delta(\Gamma_3)}{\Delta(\Gamma)} \left( \frac{1}{\Delta(\Gamma_2)}+ \frac{1}{\Delta(\Gamma_3)} -1 \right) \leq 0.
\end{align*}
Hence $e(E_0; S, \alpha) = e(E_0; T, E+f^{-1}_*(\alpha)) \geq 1$, and $E_0$ is an exceptional tiger.
By \cite[8.3.9 Lemma]{KM}, we have $e(E_0; S, \emptyset) \geq e(E; S, \emptyset)$.
Thus we can replace $E$ as $E_0$.

Suppose that $T$ is a net.
Let $F$ be a general $\pi$-fiber.
Then 
\begin{align*}
    0 \geq (f^*(K_S+\alpha) \cdot F) = (K_T+E+f^{-1}_*(\alpha) \cdot F) \geq -2+(E \cdot F).
\end{align*}
Hence we obtain the description (a).

Suppose that $\pi \colon T \to S_1$ is birational.
Let $q_1 = \pi (\Sigma) \in S_1$ and $e = e(E; S, \emptyset)$.
Since $E^2 <0 < C^2$, $E$ intersects with $\Sigma$. 
Hence $q_1 \in C$. 
By Lemma \ref{lem:Sarkisov} (4), $K_T+aE$ is $\pi$-trivial for some $a >e$.
Since $K_T+E \sim -f^{-1}_*(\alpha) +f^*(K_S+\alpha)$ is numerically equivalent to an anti-effective $\Q$-Weil divisor, so is $f_*(K_T+E)=K_S+C$.
Hence $K_S+C$ is anti-nef.

If $a \geq 1$, then $K_{S_1}+C$ is plt since so is $K_T+E$, and we obtain the description (c).
Hence we may also assume that $a<1$.
By the choice of $E$, $K_{S_1}+C$ is plt away from $q_1$.
Lemma \ref{LacA.19} now shows that $C \subset S^\circ$, $K_{S_1}+C$ is almost log canonical of one of types (a)--(c) in Definition \ref{def:almostlc} at $q_1$, or the minimal resolution of $S_1$ is a log resolution of $(S_1, C)$.

Suppose further that the first case occurs.
If $C$ is smooth, then we obtain the description (c).
If $C$ is singular, then $C$ has a singularity of genus one at $q_1$ since $2(p_a(C)-1) \leq (K_{S_1}+C \cdot C) \leq 0$.
Hence we obtain the description (b).

Suppose further that the second case occurs.
Take $E_1$ as an exceptional divisor in $\wt S_1$ over $q_1$.
Set $f(t) = e(E_1; S_1, tC)$.
Then $f(0) \geq 0$.
Moreover, $f(1) =1$ because the singularity $q_1 \in (S_1, C)$ has the same numerical property as a log canonical pair with analytically reducible boundary.
Since $f$ is an affine function, we obtain $f(a) \geq a$.
Now consider $g(t) \coloneqq e(E_1; T, tE)$.
Then $g(0) \geq 0$.
We have just seen that $g(a)=f(a) \geq a$
Hence $g(e) \geq e$.
On the other hand, define $d \coloneqq (f^*(K_S+\alpha) \cdot \Sigma)/\Sigma^2 \geq 0$.
Then 
\begin{align*}
    f^*(K_S+\alpha)  = K_T+E+f^{-1}_*(\alpha) = \pi^*(K_{S_1}+C + \pi_*f^{-1}_*(\alpha)) + d\Sigma.
\end{align*}
Since $e(E_1; S, \alpha)=e(E_1; T, E+f^{-1}_*(\alpha)) \geq e(E_1; S_1, C)=1$, the divisor $E_1$ is an exceptional tiger of $S$ in $\wt S$.
Then $g(e) = e$ by the choice of $E$.
Hence $g(t) \geq t$ for all $t \geq 0$.
However, we have $g(1) < 1$ since $(T, E)$ is plt, a contradiction.

Suppose further that the last case occurs.
If $q_1 \in S_1^\circ$, then $(S_1, C)$ is log smooth at $q_1$ by assumption and hence we have the description (c).
Hence we may assume that $q_1 \not \in S_1^\circ$.
Let $c$ be the log canonical threshold of $(S_1, C)$ at $q_1$.
Then $1 \geq c \geq a$ since $K_T+aE = \pi^*(K_{S_1}+aC)$.
Assume that there is an exceptional divisor $E_1$ in $\wt S_1$ over $q_1$ such that $e(E_1; S_1, cC) =1$.
Set $f(t) \coloneqq e(E_1; S_1, tC)$.
Then $f(0) \geq 0$ and $f(c)=1 \geq c$ by assumption.
Hence $f(a) \geq a$.
We now apply the same argument as in the second case to obtain a contradiction.
Since $q_1 \not \in S_1^\circ$, we conclude that $c=1$ and $(S, C)$ is plt at $q_1$ by \cite[8.3.3 Lemma]{KM}.
Hence we have the description (c).

Combining these results, we obtain the assertion.

\noindent(2): Let $C \subset S$ be a tiger.
Then $K_S+C$ is anti-nef by definition.
First suppose that $C \not\subset S^\circ$.
Then $(S, C)$ is log smooth around $S^\circ$ by Lemma \ref{LacA.19}.
If there is an exceptional divisor $E$ in $\wt S$ such that $e(E; S,C) \geq 1$, then $E$ is an exceptional tiger, a contradiction with the assumption.
Now \cite[8.3.3 Lemma]{KM} shows that $(S, C)$ is plt.
Next suppose that $C \subset S^\circ$.
If $p_a(C)=0$, then $K_S+C$ is plt.
Hence we may assume that $p_a(C) \geq 1$.
Then $p_a(C) = 1$, $C \sim -K_S$, and $S$ is a Du Val del Pezzo surface.
If $K_S^2\geq 8$, then $S \cong \PP^2_k$ or a quadric cone in $\PP^3_k$ and we can replace $C$ as a conic so that $(S, C)$ is plt.
Otherwise, we can replace $C$ as a $(-1)$-curve so that $C \not \subset S^\circ$ and hence $(S, C)$ is plt.
Combining these results, we obtain the assertion.
\end{proof}

In Lemma \ref{lem:1sthunt}, we will see that the case (b) of Theorem \ref{Lac6.2} (1) actually occurs.
Based on Theorem \ref{Lac6.2}, we define some classes of log del Pezzo surfaces of rank one.

\begin{defn}
\label{def:LDP(a)}
A log del Pezzo surface $S$ of rank one belongs to \textit{LDP (a)} if $S$ is constructed as follows:
Let $\pi \colon T \to \PP^1_k$ be a $\PP^1_k$-fibration of relative Picard rank one.
Take a $K_T$-non-negative smooth rational curve $E \subset T$ which is $\pi$-section or $\pi$-bisection such that $K_T+E$ is plt.
Then $S$ is the contraction of $E$. 
\end{defn}

\begin{defn}
\label{LDP(b)}
A log del Pezzo surface $S$ of rank one belongs to \textit{LDP (b)} if $S$ is constructed as follows:
Let $S_1$ be a Du Val del Pezzo surface of rank one and $C$ an irreducible, reduced and singular member of $|-K_{S_1}|$ contained in $S_1^\circ$.
Let $q_1 \in C$ be the singularity of $C$.
Take a $K_T$-negative contraction $\pi \colon T \to S_1$ of relative Picard rank one such that the exceptional divisor lies over $q_1$, the strict transform $E$ of $C$ in $T$ is $K_T$-non-negative and contractible.
Then $S$ is the contraction of $E$. 
\end{defn}

\begin{defn}
\label{LDP(c)}
A log del Pezzo surface $S$ of rank one belongs to \textit{LDP (c)} if $S$ is constructed as follows:
Let $(S_1, C)$ be one of the surface pairs in \cite[Proposition 5.1 or Lemmas 5.2--5.4]{Lac}.
Choose a point $q_1 \in C$.
Take a $K_T$-negative contraction $\pi \colon T \to S_1$ of relative Picard rank one such that the exceptional divisor lies over $q_1$, the strict transform $E$ of $C$ in $T$ is $K_T$-non-negative and contractible.
Then $S$ is the contraction of $E$. 
\end{defn}

\begin{defn}
A log del Pezzo surface $S$ of rank one belongs to \textit{LDP (d)} if $S$ contains a curve $C$ such that $(S, C)$ is one of the pairs in \cite[Proposition 5.1 (1) or Lemmas 5.2--5.4]{Lac}.
\end{defn}

Now we are ready to state an alternative of \cite[Theorem 6.25]{Lac} assuming the existence of tigers.

\begin{thm}[{cf. \cite[Theorem 6.25]{Lac}}]\label{Lac6.22}
Let $S$ be a log del Pezzo surface of rank one over an algebraically closed field of characteristic $p \neq 2$ or $3$.
Suppose that $S$ 
has a tiger.
Then $S$ belongs to one of classes LDP (a)--(d).
\end{thm}

\begin{proof}
The assertion follows from Theorem \ref{Lac6.2} and \cite[\S 5.1]{Lac}.
We note that \cite[Example 7.6]{Lac} belongs to LDP (b) by construction.
\end{proof}

\section{Proof of the main theorems}\label{sec:main}

In this section, we prove Theorem \ref{thm:main} and Theorem \ref{thm:main2} (2).
Throughout this section, we always assume that $p=5$.
First let us mention the case where there is no tiger.

\begin{lem}\label{lem:NLnotiger}
Let $S$ be a log del Pezzo surface of rank one with no tiger. Then the following are equivalent.
\begin{enumerate}
\item[\textup{(1)}] $S$ is constructed as in Example \ref{eg:Lac4.8cex-1}.
\item[\textup{(2)}] $\Dyn(S)=2[3,2]+[3]+[2]+[2^4]$, $[4,2]+[3,2^5]+[3,2]+[2]$, or $[2,4]+[2,3,2^2]+[3]+[2^4]$.
\item[\textup{(3)}] There exists no log del Pezzo surface of rank one over $\C$ with the same Dynkin type as $S$.
\item[\textup{(4)}] $S$ is not log liftable over $W(k)$.
\end{enumerate}
\end{lem}

\begin{proof}
The assertion (1) $\Rightarrow$ (2) is obvious and (3) $\Rightarrow$ (4) follows from Lemma \ref{lem:NDtoNL}.
Since each Dynkin type violates the Bogomolov bound \cite[9.2 Corollary]{KM}, the assertion (2) $\Rightarrow$ (3) holds.
We can check that the arguments as in \cite[\S 4]{Lac} hold in characteristic both zero and five except for the absence of Example \ref{eg:Lac4.8cex-1}.
Hence the assertion (4) $\Rightarrow$ (1) holds.
\end{proof}

Next, we consider the case where there is a tiger.
Let us summary consequences of Theorem \ref{Lac6.22} for non-log liftable log del Pezzo surfaces of rank one.

\begin{lem}\label{cor:NLS1}
Let $S$ be a log del Pezzo surface of rank one with a tiger which is not log liftable over $W(k)$.
Then $S$ belongs to LDP (b) and we can choose $(S_1, C)$ as the pair of the Du Val del Pezzo surface with $\Dyn (S_1) =2 [2^4]$ and the cuspidal anti-canonical member.
\end{lem}

\begin{proof}
Since the arguments as in \cite[\S\S 5.1]{Lac} hold in characteristic both zero and five, the non-log liftability shows that $S$ does not belong to LDP (d).
By Theorem \ref{Lac6.22}, $S$ belongs to LDP (a), (b), or (c).

Suppose that $S$ belongs to LDP (a).
Let $\sigma \colon \wt S \to S$ be the minimal resolution.
Then there is a $\PP^1_k$-fibration $\wt \pi \colon \wt S \to \PP^1_k$ which factors through $\pi \colon T \to \PP^1_k$ as in Definition \ref{def:LDP(a)}.
By \cite[Remark 4.2 and Lemma 4.9]{Kaw21}, we obtain $H^2(\wt S, T_{\wt S}(-\log E_\sigma))=0$.
Since $\wt S$ is rational, we have $H^2(\wt S, \sO_{\wt S})=0$.
Thus $(\wt S, E_\sigma)$ lifts to $W(k)$ by \cite[Theorem 2.3]{KN1}, a contradiction.

Hence $S$ belongs to LDP (b) or (c).
By assumption, any log resolution of the pair $(S, \emptyset)$ lifts to $W(k)$.
Then neither does the pair $(S_1, C)$ as in Definitions \ref{LDP(b)} and \ref{LDP(c)}.

Suppose $S$ belongs to LDP (c).
Then, by \cite[\S 5.1]{Lac}, $(S_1, C)$ must be the surface pairs as in \cite[Example 7.6]{Lac}.
By construction, the hunt step for $(S_1, C)$ gives the Du Val del Pezzo surface $S_2$ with Dynkin type $2[2^4]$ and its cuspidal anti-canonical divisor $C'$ contained in $S_2^\circ$.
Hence $S$ also belongs to LDP (b).

Suppose $S$ belongs to LDP (b).
Repeated blow-ups at the intersection of $C$ and a $(-1)$-curve gives an extremal rational elliptic fibration $Z$ with singular fiber $F$ of type $II$ or $I_1$.
By \cite[Theorem 4.1]{Lang2}, the non-log liftability of $(S, \emptyset)$ shows that the singular fibers of $Z$ are of type $II$, $I_5$, $I_5$. 
Hence $\Dyn (S_1) =2 [2^4]$ and $C$ is cuspidal, and the assertion holds.
\end{proof}

\begin{lem}\label{lem:NLDyn}
Let $S$ be a log del Pezzo surface of rank one.
Suppose that $S$ belongs to LDP (b) and $(S_1, C)$ is the pair of the Du Val del Pezzo surface with $\Dyn (S_1) =2 [2^4]$ and the cuspidal anti-canonical member.
Then the following hold.
\begin{enumerate}
    \item[\textup{(1)}] $\Dyn (S) = 2[2^4] + (\dagger)$, where $(\dagger)$ is listed in Table \ref{tab:dagger}.
    \item[\textup{(2)}] The isomorphism class of $S$ is uniquely determined by its Dynkin type except when $(\dagger) = [2^n]+[2+n; [2], [3], [5]]$ for some $n \geq 0$.
    In these cases, the isomorphism classes of del Pezzo surfaces correspond to the closed points of $\PP^1_k \setminus \{0, 1, \infty\}$ for each $n \geq 0$.
\end{enumerate}
\end{lem}

\begin{proof}
Take $\pi \colon T \to S_1$ and $E \subset T$ as in Definition \ref{LDP(b)}.
Then the first assertion follows from easy computation of $\pi$ and the classification of klt singularities \cite[\S 3]{Kol92}.
By \cite[23.3 Lemma]{KM}, the contraction of $E$ always gives a log del Pezzo surface of rank one.
When $(\dagger) \neq [2^n]+[2+n; [2], [3], [5]]$ for any $n \geq 0$, we see at once that $\Dyn(S)$ determines $\pi$ uniquely.
Now suppose that $(\dagger) = [2^n]+[2+n; [2], [3], [5]]$ for some $n \geq 0$.
Then $S$ also belongs to LDP (c), and we can replace $(S_1, C)$ as in \cite[Example 7.6]{Lac}.
Then $\pi$ is induced by the blow-up of at a point $q_1 \in C \cap S_1^\circ$ $(n+1)$ times along $C$.
By construction, the induced morphism $T \to S$ is the extraction of the central curve of the $[2+n; [2], [3], [5]]$-singularity.
Hence the automorphism group of $S$ equals that of $T$, and the group of automorphisms of $S_1$ fixing $q_1$.
On the other hand, since every automorphisms of $S_1$ fix the $[2]$, $[3]$, and $[5]$-singularities and $C$, they fix each point of $C \cap S_1^\circ$. 
Therefore the isomorphism classes of del Pezzo surfaces with Dynkin type $[2^n]+[2+n; [2], [3], [5]]$ are parametrized by $C \cap S_1^\circ \cong \PP^1_k \setminus \{0,1, \infty \}$.
Combining these results, we obtain the second assertion.
\end{proof}

Next let us check when the Dynkin types determine the non-log liftability.

\begin{lem}\label{lem:ND}
Let $S$ be a log del Pezzo surface of rank one.
Suppose that $\Dyn (S) = 2[2^4] + (\dagger)$, where $(\dagger)$ is listed in Table \ref{tab:dagger} but $(\dagger) \neq [3]$ or $[2, 4]$.
Then there exists no log del Pezzo surface of rank one over $\C$ with the same Dynkin type as $S$.
\end{lem}

\begin{proof}
If $(\dagger) \neq [2; [2], [3], [5]]$, then $\Dyn (S)$ violates the Bogomolov bound \cite[9.2 Corollary]{KM}, and the assertion holds.
Hence we may assume that $(\dagger) = [2; [2], [3], [5]]$.

Suppose by contradiction that there is a log del Pezzo surface $S_{\C}$ of rank one over $\C$ with $\Dyn (S_{\C})=2[2^4]+[2; [2], [3], [5]]$.
Run the hunt step for $(S_{\C}, \emptyset)$, i.e., take $f_0 \colon T_1 \to S_{\C}$ as the extraction of the central curve $E_1$ of $[2; [2], [3], [5]]$ and $\pi_1 \colon T_1 \to S_{1}$ be the $K_{T_1}$-negative contraction.
From now on, we follow the notation of \cite[8.2.10]{KM}.
An easy computation shows that $e_0=28/29<a_1$.

Assume that $T_1$ is a net.
Let $F$ be a general $\pi_1$-fiber.
Since $0=(K_{T_1}+a_1 E_1 \cdot F) > -2+28/29(E_1 \cdot F)$, we have $(E_1 \cdot F)=1$ or $2$.
If the former holds, then the $\pi_1$-fiber intersecting with $E_1$ at the $[2]$-singularity must contain another $[2]$-singularity by \cite[3.4 Lemma]{KM}, a contradiction.
Hence the latter holds.
Since any $\pi_1$-fibers cannot contain two of $[2]$, $[3]$ and $[5]$-singularities, there are three distinct $\pi_1$-fibers $F_2$, $F_3$, and $F_5$ such that each $F_i$ passes through the $[i]$-singularity.
Since these fibers are multiple, $F_i \cap E_1$ supports on the $[i]$-singularity for $i=2,3,5$.
It implies, however, that the double covering $\pi_1|_{E_1}$ ramifies at three points, a contradiction.

Hence $\pi_1$ must be birational.
Next let us show that $(S_1, A_1)$ is plt.
If $a_1 \geq 1$, then $(S_1, A_1)$ is plt since so is $(T, E_1)$.
Hence, to this end, we may assume that $a_1<1$.
Then $(S_1, a_1 A_1)$ is flush by Lemma \ref{lem:flush}.
If $q_1 \in S_1$ is singular, then $(S_1, A_1)$ is plt by Lemma \ref{lem:singflush}.
Suppose that $q_1 \in S_1$ is smooth.
Then \cite[8.3.7 Lemma (4)]{KM} shows that $A_1$ is normal crossing at $q_1$ since $a_1 \geq 4/5$.
Hence $A_1$ is either smooth or has a simple node at $q_1$.
Assume that the latter case holds.
By \cite[11.1.1 Lemma]{KM} and $a_1  > 1/2$, $E_1 \cap \Sigma_1$ consists of either one smooth point and the $[2]$-singularity, or two singular points.
In the former case, we obtain $a_1=2/3 < 28/29$, a contradiction.
In the latter case, $E_1 \cap \Sigma_1$ consists of the $[2]$ and $[3]$-singularities since $S_1$ is smooth at $q_1$.
However, $(K_T+a_1 E_1 \cdot \Sigma_1)=0$ gives $a_1=4/5 < 28/29$, a contradiction.

Hence $A_1$ is smooth at $q_1$ and $(S_1, A_1)$ is plt.
If $\Sigma_1$ contains no $[2^4]$-singularities, then $(S_1, A_1)$ violates the Bogomolov bound \cite[9.2 Corollary]{KM}.
From this fact, $\Sigma_1$ contains one $[2^4]$-singularity.
By the contractibility of $\Sigma_1$, this point is the unique singularity of $T$ which $\Sigma_1$ passes through.
Hence $\Dyn (S_1)=[2^4]+[2]+[3]+[5]$ and $A_1$ passes through $[2]$, $[3]$, and $[5]$-singularities.
In particular, $K_{S_1}+A_1$ is anti-ample.
However, this contradicts \cite[23.5 Proposition]{KM} since $[2^4]$ is a cyclic singularity.
Hence the assertion holds.
\end{proof}

On the other hand, $2[2^4]+[3]$ and $2[2^4]+[2, 4]$ are feasible over $\C$ as follows.

\begin{eg}\label{eg:nonND}
We follow the notation of Example \ref{eg:2A4} with $k=\C$.
Let $\alpha$ be a solution to $t^2+11t-1=0$.
By blowing up the node of $C_\alpha$ once (resp.\ twice along one of two branches of $C_\alpha$), the strict transform of $C_\alpha$ becomes a $(-3)$-curve (resp.\ a $(-4)$-curve), and hence we get the minimal resolution of a log del Pezzo surface of rank one whose Dynkin type is $2[2^4]+[3]$ (resp.\ $2[2^4]+[2, 4]$).
\end{eg}

To prove Theorem \ref{thm:main}, we have to treat the case where the Dynkin type is $2[2^4]+[3]$ or $2[2^4]+[2, 4]$.
The next lemma shows that such a log del Pezzo surface is always constructed as in Example \ref{eg:nonND}.

\begin{lem}\label{lem:1sthunt}
Let $S$ be a log del Pezzo surface of rank one with $\Dyn (S) =2 [2^4] +[3]$ or $2[2^4] +[2, 4]$.
Let $(S_1, \Delta_1)$ be the outcome of the hunt step for $(S, \emptyset)$ and $A_1$ the support of $\Delta_1$.
Then $S_1$ is a Du Val del Pezzo with Dynkin type $2[2^4]$ and $A_1$ is a singular anti-canonical member of $S_1$.
\end{lem}

\begin{proof}
We follow the notation of \cite[8.2.10]{KM}.
Since $T_1$ is a Du Val surface containing at least $2[2^4]$ singularities, \cite[3.4 Lemma]{KM} shows that $T_1$ is not a net. 
Hence $\pi_1 \colon T_1 \to S_1$ is birational.

By the contractibility of the $\pi_1$-exceptional divisor $\Sigma_1$, it contains at most one singularity and $K_{T_1}+\Sigma_1$ is plt.
Then $\Dyn (S_1)$ is one of $[2]+2[2^4]$, $2[2^4]$, $[2]+[2^4]$ or $[2^4]$.
Comparing \cite[Theorem B.7]{Lac}, we have $\Dyn (S_1) \neq [2]+2[2^4]$ or $[2]+[2^4]$.
In particular, $q_1 \in A_1 \subset S_1^\circ$.
If $a_1 \geq 1$, then $K_{S_1}+A_1$ is anti-ample and hence $A_1$ is a smooth rational curve in $S_1^\circ$, a contradiction with \cite[13.7 Lemma]{KM}.
Hence $a_1 < 1$ and $(S_1, a_1 A_1)$ is flush by Lemma \ref{lem:flush}.
Since $1/3 \leq e_0<a_1$, \cite[8.3.7 Lemma (1)]{KM} shows that the multiplicity of $A_1$ at $q_1$ is at most three.
By \cite[13.7 Lemma]{KM}, the multiplicity cannot be one.
Since $T_1$ is Du Val and $E_1 \cap \Sing T_1$ consists of at most one $[2]$-singularity, \cite[\S 11.1--\S 11.3]{KM} shows that the multiplicity is two and $T_1$ has type I or II.

Assume that $\Dyn (S_1) =[2^4]$.
Then $\Dyn (S) =2[2^4]+[3]$, $T_1$ has type I and $\Sigma_1$ contains one $[2^4]$-singularity.
In particular, $A_1$ has a double point of genus five at $q_1$.
On the other hand, write $A_1 \sim -nK_{S_1}$ for some $n \in \Z_{>0}$.
Then the genus formula yields $5 = p_a(A_1) = \frac12 (A_1+K_{S_1} \cdot A_1)+1=\frac52 n(n-1)+1$.
However, this has no integer solution, a contradiction.

Hence $\Dyn (S_1) =2[2^4]$.
If $\Dyn (S) =2[2^4]+[3]$, then $T_1$ has type I and $\Sigma_1 \subset T_1^\circ$.
If $\Dyn (S) =2[2^4]+[2, 4]$, then $T_1$ has type II and $\Sigma_1 \cap \Sing T_1$ contains the $[2]$-singularity.
In each case, $A_1$ has a double point of genus one at $q_1$.
Hence $A_1$ is a singular member of $|-K_{S_1}|$, and the assertion holds.
\end{proof}

\begin{rem}\label{rem:1sthunt}
We follow the notation of the proof of Lemma \ref{lem:1sthunt}.
\begin{enumerate}
    \item[\textup{(1)}] Let $f_0 \colon T_1 \to S$ be the extraction of $E_1$ and $\overline{\Sigma_1} = f_0(\Sigma_1) \subset S$.
    Then we see at once that $K_S+\overline{\Sigma_1}$ is numerically trivial and $e(E_1; S, \overline{\Sigma_1})=1$.
    Hence $E_1$ is a tiger and $S$ belongs to the case (b) of Theorem \ref{Lac6.2}.
    Suppose that $p=5$ in addition.
    Then $q_1 \in A_1$ is a cusp as we have seen in Example \ref{eg:2A4}. 
    Hence$(S, \overline{\Sigma_1})$ is almost log canonical of type (a) or (c) in Definition \ref{def:almostlc} at the point $f_0(E_1)$.
    \item[\textup{(2)}] The proof of Lemma \ref{lem:1sthunt} is characteristic-free.
    It also holds that the strict transform $D$ of $\Sigma_1$ in $\wt S$ is characterized by the intersection numbers with $-K_{\wt S}$ and the irreducible components of $E_\sigma$, where $\sigma \colon \wt S \to S$ is the minimal resolution.
\end{enumerate}

\end{rem}

\begin{prop}\label{prop:NL}
Let $S$ be a log del Pezzo surface of rank one with $\Dyn (S) =2 [2^4] +[3]$ or $2[2^4] +[2, 4]$.
Then $S$ is not log liftable over $W(k)$.
\end{prop}

\begin{proof}
Let $\sigma \colon \wt S \to S$ be the minimal resolution.
Suppose by contradiction that $S$ is log liftable over $W(k)$.
Let $E_\sigma = \sum_{i=1}^n E_i$ be the irreducible decomposition and take $(\wt{\mc{S}}, \mc{E} \coloneqq \sum_{i=1}^n \mc{E}_i)$ as a $W(k)$-lifting of $(\wt S, E_\sigma)$.
Let $K$ be the fractional field of $W(k)$ and for a field extension $K \subset L$, write $\wt{S}_L \coloneqq \wt{\mc{S}} \otimes_{W(k)} L$ and $E_{i,L} \coloneqq \mc{E}_i \otimes_{W(k)} L$ for each $i$.
Let $\sigma_\C \colon \wt S_\C \to S_\C$ be the contraction of $E_\C$.
Take $D \subset \wt S$ and $D_\C \subset \wt {S}_\C$ as in Remark \ref{rem:1sthunt}. 
Since $(D_\C \cdot -K_{\wt{S}_\C}) = (D \cdot -K_{\wt S})$ and $(D_\C \cdot E_{i, \C}) = (D \cdot E_i)$ for each $1 \leq i \leq n$,  $D_\C$ descends to $\Spec K$ and the closure $\mc{D}$ of $D_\C$ in $\wt{\mc{S}}$ is a lift of $D$.

On the other hand, $D$ spans an extremal lay of the Mori cone $\overline{\mathrm{NE}}(\wt S)$.
Since every line bundle on $\wt S$ is liftable to $\wt {\mc{S}}$ by \cite[Corollary 8.5.6 (a)]{FAG}, $D$ also spans an extremal lay $R$ of $\overline{\mathrm{NE}}(\wt{\mc{S}}/\Spec W(k))$.
By \cite[Proposition 5.7]{TY}, there is a contraction $\Phi \colon \wt{\mc{S}} \to \wt{\mc{S}}_1$ of $R$.
Since the special fiber of $\wt{\mc{S}}_1$ is the contraction of $D$, which is smooth, $\wt{\mc{S}}_1$ is also smooth over $\Spec W(k)$.
Since $\Phi$ contracts the special fiber of $\mc{D}$ to a point, $\Phi(\mc{D})$ is a section over $\Spec W(k)$.

Now suppose that $\Dyn (S)=2[2^4]+[2, 4]$.
Then there are exactly two irreducible components of $\mc{E}$, say $\mc{E}_1$ and $\mc{E}_2$ which intersects with $\mathcal{D}$ as we have seen in the proof of Lemma \ref{lem:1sthunt}.
We may assume that $\mc{E}_2$ is the lift of the $(-4)$-curve.
Then $\mc{F} \coloneqq \sum_{i=2}^n \Phi(\mc{E}_i)$ is a simple normal crossing divisor over $\Spec W(k)$ and $(\wt{\mc{S}}_1, \mc{F})$ is a log lift of a log del Pezzo surface with Dynkin type $2[2^4]+[3]$.

Hence we may assume that $\Dyn (S)=2[2^4]+[3]$ to prove the assertion.
Then there is a unique irreducible component of $\mc{E}$, say $\mc{E}_1$ which intersects with $\mathcal{D}$ as we have seen in the proof of Lemma \ref{lem:1sthunt}.
Then $(\wt{\mc{S}}_1, \mc{F} \coloneqq \sum_{i=2}^n \Phi(\mc{E}_i))$ is a log lift of the Du Val del Pezzo surface $S_1$ with $\Dyn (S_1) = 2[2^4]$.
Lemma \ref{lem:1sthunt} also shows that $\mc{A}_1 \coloneqq \Phi(\mc{E}_1)$ defines a singular anti-canonical member in both the special fiber and the generic fiber.
As the generic fiber, we get a weak del Pezzo surface $S_{1,K}$, the union of all $(-2)$-curves $F_K = \mc{F} \otimes_{W(k)} K$, and one of two nodal anti-canonical members $A_{1,K} = \mc{A}_1 \otimes_{W(k)} K$.
In particular, the other nodal anti-canonical member of $S_{1, K}$, say $B$, is also defined over $K$. 
Let $Z_K$ be the blow-up of $S_{1, K}$ at the basepoint of $|-K_{S_{1, K}}|$, which is a $K$-rational point.
Then $|-K_{Z_K}|$ defines an elliptic fibration $g \colon Z_K \to \PP^1_K$.
Let $G$ be the strict transform $F_K \cup A_{1, K} \cup B$.
Then $g(G) \subset \PP^1_K$ consists of four $K$-rational points and $g^{-1}(g(G))$ is a disjoint union of four singular $g$-fibers.
We fix coordinates $s, t$ of $\PP^1_K$ such that $g(G)=\{st(t-s)(t-\alpha s)=0\}$ for some $\alpha \in k \setminus \{0,1\}$.

On the other hand, $g(G)$ is projectively equivalent to $\{st(t^2+11st-s^2)=0\}$ by Example \ref{eg:2A4}.
An easy computation shows that $\alpha$ is a solution of one of $x^2-123x+1=0$, $121x^2-121x-1=0$, or $x^2+121x-121=0$.
Since we can transform those equations as
\begin{align*}
    \left( \frac{1}{55} (2x-123)\right)^2=5, \left(\frac{11}{5} (2x-1) \right)^2=5, \text{ and }  \left( \frac{1}{55} (2x+121) \right)^2=5
\end{align*}
respectively, there is a solution of the equation $y^2=5$ in $K$, a contradiction to the Eisenstein's criterion.
Hence the assertion holds.
\end{proof}

Now we can prove Theorem \ref{thm:main}.

\begin{proof}[Proof of Theorem \ref{thm:main}]
If $S$ has no tiger, then the assertion holds from Lemma \ref{lem:NLnotiger}.
Hence we may assume that $S$ has a tiger.
By Lemmas \ref{cor:NLS1} and \ref{lem:NLDyn}, to show the assertions (1) and (2), it suffices to show that $S$ is not log liftable over $W(k)$ if $\Dyn(S)=2[2^4]+(\dagger)$, where $(\dagger)$ is listed in Table \ref{tab:dagger}.
When $(\dagger) \neq [3]$ or $[2,4]$ (resp.\ $(\dagger) = [3]$ or $[2,4]$), we get the assertions by Lemmas \ref{lem:NDtoNL} and \ref{lem:ND}. (resp.\ Proposition \ref{prop:NL}).
The assertion (3) follows from Lemma \ref{lem:ND} and Example \ref{eg:nonND}.
\end{proof}

Before proving Theorem \ref{thm:main2} (2), let us mention some properties of Example \ref{eg:Lac4.8cex-1}.

\begin{cor}
Let $S$ be a log del Pezzo surface of rank one as in Example \ref{eg:Lac4.8cex-1}. Then the following hold.
\begin{enumerate}
    \item[\textup{(1)}] $S$ has no tigers.
    \item[\textup{(2)}] There is a birational morphism from the minimal resolution $\wt S$ of $S$ to the Du Val del Pezzo surface with Dynkin type $2[2^4]$.
\end{enumerate}
\end{cor}

\begin{proof}
(1): If $S$ has a tiger, then Lemmas \ref{cor:NLS1} and \ref{lem:NLDyn} shows that $S$ contains $2[2^4]$-singularities, a contradiction with Example \ref{eg:Lac4.8cex-1}.

\noindent (2): Let $T$ be the blow-up of $\PP^2_k$ as in Example \ref{eg:Lac4.8cex-1} with $(n_t, n_u)=(5,2)$.
Then there is the induced birational morphism $\wt S \to T$ and we can check that $T$ is the minimal resolution of a log del Pezzo surface with Dynkin type $2[2^4]+[3]$.
Hence the assertion follows from Lemma \ref{lem:1sthunt}.
\end{proof}

To prove Theorem \ref{thm:main2} (2), we need further investigations of log del Pezzo surfaces of rank one with Dynkin type $2 [2^4] +[3]$ or $2[2^4] +[2, 4]$.
Let us show such a surface does not satisfy (NB) or (NK).

\begin{lem}\label{lem:notNB}
Let $S$ be a log del Pezzo surface of rank one with $\Dyn (S) =2 [2^4] +[3]$ or $2[2^4] +[2, 4]$.
Let $r$ be the index of $S$.
Then there is a smooth member in $|-rK_S|$.
\end{lem}

\begin{proof}
We follow the notation of \cite[8.2.10]{KM}.
When $\Dyn (S) =2 [2^4] +[3]$, we have $f_0^*(-3K_S) \sim -3K_{T_1} -E_1 \sim 3(\pi_1^*(-K_{S_1})-\Sigma_1)-(\pi_1^*(A_1)-2\Sigma_1) \sim \pi_1^*(-2K_{S_1})-\Sigma_1$.
On the other hand, when $\Dyn (S) =2 [2^4] +[2, 4]$, we have $f_0^*(-7K_S) \sim -7K_{T_1} -4E_1 \sim 7(\pi_1^*(-K_{S_1})-2\Sigma_1)-4(\pi_1^*(A_1)-3\Sigma_1) \sim \pi_1^*(-3K_{S_1})-2\Sigma_1$.
Hence, to show the assertion, it suffices to find a smooth member $D_i \in |-iK_{S_1}|$ such that $D_i \cap A_1$ supports on the cusp of $A_1$ for $i=2$ and $3$.

By Lemma \ref{lem:1sthunt} and Example \ref{eg:2A4}, $S_1$ is defined by $y^2-(x^3+2t^4x+4s^5t+2t^6)=0$ in $\PP(1,1,2,3)$, where $s$, $t$, $x$, and $y$ are coordinates of weight $1$, $1$, $2$, and $3$ respectively.
Moreover, $A_1$ is defined by $t=0$ and its cusp is $[1:0:0:0]$.
Then we can choose the desired $D_2$ (resp.\ $D_3$) as the curve defined by $x-t(s+2t)=0$ (resp. $y-t(x+t^2+st+3s^2)=0$), and the assertion holds.
\end{proof}

\begin{lem}\label{lem:notNK}
Let $S$ be a log del Pezzo surface of rank one with $\Dyn (S) =2 [2^4] +[3]$ or $2[2^4] +[2, 4]$.
Then $H^1(S, \sO_S(-A))=0$ for each ample $\Z$-Weil divisor $A$.
\end{lem}

\begin{proof}
Let $A$ be an ample $\Z$-Weil divisor on $S$.
Suppose that $\Dyn (S) =2 [2^4] +[3]$.
Then $S$ is of index $3$ and $K_S^2=1/3$.
Since $p=5 > 3 \cdot (3-1) \cdot K_S^2=2$, Lemmas \ref{lem:NKtoNB} and \ref{lem:notNB} yield $H^1(S, \sO_S(-A))=0$. 

Now suppose that $\Dyn (S) = 2[2^4] +[2, 4]$.
Then $S$ is of index $7$ and $K_S^2=1/7$.
From now on, we follow the notation of Example \ref{eg:2A4} and Lemma \ref{lem:1sthunt}.
Then $U_2$ is the minimal resolution of $S_1$ and $C_2$ is the strict transform of $A_1$ in $\PP^2_k$.
For abbreviation, we continue to write $C_2$ for its strict transform in $U_2$

Now take $\wt \pi_1 \colon \wt S \to U_2$ as the blow-up of $U_2$ at the cusp of $C_2$ two times along $C_2$.
Then Lemma \ref{lem:1sthunt} shows that $\wt S$ is the minimal resolution of $S$.
There are two $\wt \pi_1$-exceptional curves $G_1$ and $G_2$ such that $G_1^2=-2$ and $G_2^2=-1$.
To simplify notation, we continue to write $L_{ab}, \ldots, L_{cd}, E_{a}, \ldots, E_d, F_a, \ldots, F_d, C_2$ for the strict transform of these curves in $\wt S$.
Note that the minimal resolution $\sigma \colon \wt S \to S$ is the contraction of $L_{ab}$, $L_{bc}$, $L_{cd}$, $L_{ad}$, $E_a$, $E_b$, $E_c$, $E_d$, $G_1$, and $C_2$.
We denote $L_{ac, V}, L_{bd, V}, F_{a, V}, \ldots F_{d, V}$ and $G_{2, V}$ the strict transform of $L_{ac}, L_{bd}, F_a, \ldots, F_d$ and $G_2$ in a birational model $V$ of $S$ respectively.
By construction, the class divisor group of $S$ is generated by $L_{ac, S}, L_{bd, S}, F_{a, S}, \ldots F_{d, S}$ and $G_{2, S}$.
Easy calculations gives that 
\begin{align}\label{eq:[4,2]}
\begin{split}
    (C_2 \cdot L_{ac})&=(C_2 \cdot L_{bd})=(C_2 \cdot F_{a}) = \cdots = (C_2 \cdot F_{d}) = 1, \\
    (G_1 \cdot L_{ac})&=(G_1 \cdot L_{bd})=(G_1 \cdot F_{a}) = \cdots = (G_1 \cdot F_{d}) = 0, \\
    (-K_S \cdot L_{ac, S})&=(-K_S \cdot L_{bd, S})=(-K_S \cdot F_{a, S}) = \cdots = (-K_S \cdot F_{d, S}) = 3/7, \\
    (C_2 \cdot G_2)&=(G_1 \cdot G_2)=1, (-K_S \cdot G_{2, S}) = 1/7.
\end{split}
\end{align}

Now let us show the assertion.
Since $(-K_S \cdot A) \in \frac17 \Z$, we may assume that $(-K_S \cdot A)=1/7$ by Lemmas \ref{lem:NKtoNB} and \ref{lem:notNB}.
Then we can write 
\begin{align*}
    A &= G_{2, S}+n_{ac}L_{ac, S}+n_{bd}L_{bd, S}+n_aL_{a, S}+ \cdots + n_d L_{d, S}
\end{align*} 
for some $n_{ac}, n_{bd}, n_a, \ldots, n_d \in \Z$ such that $n_{ac}+n_{bd}+n_a+ \cdots +n_d=0$. 
By \cite[Lemma 3.3]{KN1}, we have $H^i(S, \sO(-A))=H^i(\wt S, \sO_{\wt S}(- \lceil \sigma^* A \rceil ))$ for $i \geq 0$.
By (\ref{eq:[4,2]}), coefficients of $\sigma^*(n_{ac}L_{ac, S}+n_{bd}L_{bd, S}+n_aL_{a, S}+ \cdots + n_d L_{d, S})$ along $C_2$ and $G_1$ equal zero.
On the other hand, we have $\sigma^*G_{2,S} = G_2+ 5/7 G_1 + 3/7 C_2$.
Hence
\begin{align*}
    &\lceil \sigma^* A \rceil \\
    =& \lceil \sigma^* G_{2, S} \rceil + \lceil \sigma^*(n_{ac}L_{ac, S}+n_{bd}L_{bd, S}+n_aL_{a, S}+ \cdots + n_d L_{d, S}) \rceil \\
    =& C_2 + G_1 + G_2 + \lceil \tau^*(n_{ac}L_{ac, S_1}+n_{bd}L_{bd, S_1}+n_aL_{a, S_1}+ \cdots + n_d L_{d, S_1}) \rceil \\
    =& - G_1 - 2 G_2 + \lceil \tau^*(C_{2, S_1}+n_{ac}L_{ac, S_1}+n_{bd}L_{bd, S_1}+n_aL_{a, S_1}+ \cdots + n_d L_{d, S_1}) \rceil,
\end{align*}
where $\tau \colon \wt S \to S_1$ is the induced morphism.
Set $A' = C_{2, S_1}+ n_{ac}L_{ac, S_1}+n_{bd}L_{bd, S_1}+n_aL_{a, S_1}+ \cdots + n_d L_{d, S_1}$.
By the Riemann-Roch theorem, we have 
\begin{align*}
    &\chi (\wt S, - \lceil \sigma^* A \rceil)\\
    = &1+ \frac12( G_1 + 2G_2 - \lceil \tau^* A' \rceil \cdot -K_{\wt S} + G_1 + 2G_2 - \lceil \tau^*A' \rceil)\\
    = &1 + \frac12( - \lceil \tau^* A' \rceil \cdot -K_{\wt S} - \lceil \tau^*A' \rceil)\\
    = &\chi (\wt S, - \lceil \tau^* A' \rceil).
\end{align*}
\cite[Lemma 3.3]{KN1} now yields $\chi(S, -A)=\chi(S_1, -A')$.
By \cite[Theorem 1.7 (3)]{KN1}, the right hand side equals $h^2(S_1, \sO_{S_1}(-A'))$.
By the Serre duality for Cohen-Macaulay sheaves, we obtain
\begin{align*}
    -h^1(S, \sO_S(-A))+h^0(S, \sO_S(K_S+A)) = h^0 (S_1, \sO_{S_1}(K_{S_1}+A')),
\end{align*}
which implies that $h^1(S, \sO_S(-A))=0$ unless $h^0(S, \sO_S(K_S+A))>0$.

Now suppose that $h^0(S, \sO_S(K_S+A))>0$.
Then $A \sim -K_S \sim G_{2,S}$ since $(-K_S \cdot K_S+A)=0$.
In particular, we may assume that $n_{ac}=n_{bd}=n_a= \cdots =n_d=0$.
Then $A' \sim C_{2, S_1} \sim -K_{S_1}$.
In particular, $h^0(S, \sO_S(K_S+A)) = h^0 (S_1, \sO_{S_1}(K_{S_1}+A'))=1$.
Therefore $h^1(S, \sO_S(-A))=0$ and the assertion holds.
\end{proof}

Now we can prove Theorem \ref{thm:main2}.

\begin{proof}[The proof of Theorem \ref{thm:main2}]
The assertion (1) follows from Lemma \ref{lem:NDtoNL}.
To show the assertion (2), it suffices to consider log del Pezzo surfaces which are not log liftable over $W(k)$ by \cite[Proposition 3.4]{KN1}.
Theorem \ref{thm:main} now shows that these Dynkin types are either $2[2^4]+[3]$ or $2[2^4]+[2,4]$.
Hence the assertion (2) follows from Lemma \ref{lem:notNK}.
\end{proof}

\section*{Acknowledgements}
The author is indebted to Tatsuro Kawakami for several helpful discussions and comments on log liftability.
Discussions with Shou Yoshikawa and Makoto Enokizono on the minimal model program for log lifts have been insightful.
The author woulWitad like to thank Justin Lacini and Hiromichi Takagi for useful comments on this paper.
The author is supported by JSPS KAKENHI Grant Number JP21K13768.

\newcommand{\etalchar}[1]{$^{#1}$}


\end{document}